\documentclass[12pt]{article}

\usepackage[T2A]{fontenc}
\usepackage[cp1251]{inputenc}
\usepackage[english,russian]{babel}

\usepackage{amsmath, amsthm}
\usepackage{amssymb}
\usepackage{amsfonts}

\theoremstyle{plain}
\newtheorem{theorem}{Теорема}
\newtheorem{lemma}{Лемма}

\theoremstyle{definition}

\usepackage{graphicx, srcltx, color, dsfont}

\def\No{No.}

\def\iu{\mathrm{i}}
\def\tht{\theta}
\def\Om{\Omega}
\def\om{\omega}
\def\e{\varepsilon}
\def\g{\gamma}

\def\l{\lambda}
\def\p{\partial}
\def\D{\Delta}

\def\a{\alpha}
\def\b{\beta}

\def\d{\delta}
\def\L{\Lambda}
\def\z{\zeta}

\def\iu{\mathrm{i}}

\def\vp{\varphi}
\def\vw{\varpi}
\def\vr{\varrho}

\def\vk{\varkappa}

\def\Ups{\Upsilon}

\def\Ho{\mathring{W}_2}
\def\cA{\mathcal{A}}

\def\cP{\mathcal{P}}

\def\rf{\mathrm{f}}

\def\rL{T}

\def\bw{\mathbf{w}}

\def\cL{\mathcal{L}}

\def\cB{\mathcal{B}}

\DeclareMathOperator{\RE}{Re}
\DeclareMathOperator{\IM}{Im}
\DeclareMathOperator{\dist}{dist}
\DeclareMathOperator{\dvr}{div}
\DeclareMathOperator{\supp}{supp}
\def\bn{\mathbf{n}}%\def\bm{\mathbf{m}}

\allowdisplaybreaks

\usepackage{dsfont}

\begin{document}

\title{Равномерная сходимость для задач с перфорацией вдоль заданного многообразия и третьим нелинейным краевым условием \\ на границах полостей}

\author{Д. И. Борисов$^{1,2,3}$, А. И. Мухаметрахимова$^{4}$}

\maketitle

\begin{quote}

{\small 1) Институт математики с ВЦ УФИЦ РАН, 
450008, Уфа, ул.
Чернышевского, 112,
Россия}

{\small 2) Башкирский государственный университет,
450076, Уфа, ул. Заки Валиди, 32,
Россия}

{\small 3) Университет Градца Кралове,
500 03, Градец Кралове,
ул. Рокитанскего, 62,
Чехия}

{\small 4) Башкирский государственный
педагогический университет им. М. Акмуллы, 
450000, Уфа,
ул. Октябрьской революции, 3а,
Россия}

{\small Emails: borisovdi@yandex.ru, albina8558@yandex.ru}
\end{quote}
%\address{
%  \\
%Башкирский государственный университет
%\\
%450076, Уфа,
%ул. Заки Валиди, 32,
%Россия
%\\
%Университет Градца Кралове,
%\\
%500 03, Градец Кралове,
%\\
%ул. Рокитанскего, 62,
%Чехия}

%\email{borisovdi@yandex.ru}

%второй автор 

%\address{Башкирский государственный
%\\
%педагогический университет им. М. Акмуллы
%\\
%450000, Уфа,
%ул. Октябрьской революции, 3а,
%\\
%Россия}
%\email{albina8558@yandex.ru}

%\subjclass[2010]{Primary 35B27; Secondary 35J65; 35P05}

%\keywords{перфорированная область, краевая задача, нелинейное третье %краевое условие, усреднение, равномерная сходимость, оценка скорости %сходимости}

\begin{abstract}
В работе рассматривается краевая задача для эллиптического уравнения второго порядка с переменными коэффициентами в многомерной области, перфорированной малыми полостями, часто расположенными вдоль заданного многообразия. Предполагается, что размеры всех полости одного порядка малости, а их форма и распределение вдоль многообразия произвольные. На границах полостей ставится третье нелинейное граничное условие. Доказана сходимость решения возмущённой задачи к решению усреднённой в нормах $L_2$ и $W_2^1$  равномерно по $L_2$-норме правой части уравнения и получены оценки скорости сходимости.

\medskip
%
%Ключевые слова: перфорированная область, краевая задача, нелинейное %третье краевое условие, усреднение, равномерная сходимость, оценка %скорости сходимости
%
%Библиография: 38 наименований.
\end{abstract}

%\thanks{Исследование  выполнено за счет гранта Российского научного %фонда
%(проект № 20-11-19995).}

\date{}%07/Feb/2022}

\section{Введение}\label{s1}

Краевые задачи в областях, перфорированных вдоль заданного многообразия, изучались во многих работах, см., например, статьи \cite{1}, \cite{2}, \cite{3}, \cite{4}, \cite{5}, \cite{6}, \cite{7}, \cite{8}, \cite{9}, \cite{10}, \cite{35},  монографии \cite{ShBook}, \cite{MarKhr}, а также списки литературы в цитированных работах. Перфорация в них описывалась малыми полостями, расположенными вдоль заданного многообразия или границы области. В задачах выделялись два малых параметра -- размеры полостей и расстояние между ними. Целью являлось изучение поведения рассматриваемых задач при уменьшении малых параметров. Основные полученные результаты  -- доказательство сходимости решений рассматриваемых задач в нормах пространств $L_2$ или $W_2^1$ к решениям некоторых усреднённых задач. При этом последние задачи отличались от исходных тем, что в них уже отсутствует перфорация, а вместо нее возникает усреднённое краевое условие на многообразии или границе области, вдоль которого располагались полости.

Упомянутые выше классические результаты о сходимости решений означают сильную или слабую резольвентную сходимость. В последние 15 лет в теории усреднения развивается новое направление исследований: появились работы, в которых для задач с быстро осциллирующими коэффициентами доказывается более сильный тип сходимости -- равномерная резольвентная сходимость, см. \cite{11}, \cite{12}, \cite{13}, \cite{14}, \cite{Sen1}, \cite{Sen2}, \cite{Pas1}, \cite{Pas2}, другие работы цитированных авторов и списки литературы в этих работах.
Для задач теории граничного усреднения вопросы равномерной резольвентной сходимости изучались в работах \cite{15}, \cite{16}, \cite{17}, \cite{18}, \cite{19}, \cite{20}, \cite{21}, \cite{22}, \cite{23}, \cite{24}, \cite{25}. В \cite{15}, \cite{16}, \cite{17}, \cite{18}, \cite{19} исследованы эллиптические операторы в плоской бесконечной полосе с частой периодической и непериодической сменой граничных условий. В \cite{21}, \cite{22}, \cite{23} рассмотрен эллиптический оператор в произвольной многомерной области с частым непериодическим чередованием граничных условий.  В \cite{20} изучен общий эллиптический самосопряженный оператор в полосе с быстро осциллирующей границей. Результаты работ \cite{15}, \cite{16}, \cite{17}, \cite{18},  \cite{19}, \cite{20}, \cite{21}, \cite{22}, \cite{23} утверждают наличие равномерной резольвентной сходимости возмущённых оператор к некоторым усреднённым и дают оценки скорости сходимости.

В \cite{24} исследован эллиптический оператор второго порядка с переменными коэффициентами в плоской полосе, перфорированной вдоль заданной кривой. На границах полостей выставлялось одно из классических краевых условий. Изучены различные возможные усредненные операторы, вид которых зависел от распределения полостей и соотношения между размерами полостей и расстояний между ними. Во всех случаях была доказана равномерная резольвентная сходимость возмущённого оператора к усреднённому и получены оценки скорости сходимости.

В \cite{25} рассматривалась краевая задача для эллиптического уравнения второго порядка с переменными коэффициентами в многомерной области, перфорированной малыми полостями вдоль заданного многообразия. Отверстия были поделены на два множества. На границах полостей первого множества ставилось условие Дирихле, на границах полостей второго множества -- третье нелинейное граничное условие. Изучался случай, когда при усреднении на многообразии возникает условие Дирихле. Была доказана сходимость решения возмущённой задачи к решению усреднённой в норме $W_2^1$ равномерно по правой части уравнения и получена неулучшаемая по порядку оценка скорости сходимости. Также было построено полное асимптотическое разложение решения возмущённой задачи в случае, когда полости образуют периодическое множество, расположенное вдоль заданной гиперплоскости.

В настоящей работе мы продолжаем исследование краевых задач в областях с непериодической перфорацией вдоль заданного многообразия, начатое в   \cite{25}. Рассматривается краевая задача для эллиптического уравнения второго порядка с переменными коэффициентами в области, перфорированной вдоль заданного многообразия. Область может быть как ограниченной, так и неограниченной. Предполагается, что все полости имеют размеры одного порядка, а форма полостей и их распределение вдоль многообразия могут быть произвольными. На границах полостей ставится третье нелинейное граничное условие. В отличии от работы \cite{25}, краевое условие Дирихле на границе полостей не выставляется. В зависимости от соотношения между размерами полостей и расстояний между ними в пределе возникают два основных случая. А именно, в первом случае при усреднении полости пропадают вместе с многообразием, вдоль которого они расположены; во втором случае при усреднении на многообразии возникает граничное условие, которое уместно интерпретировать как нелинейное дельта-взаимодействие. Нашим основным результатом работы является доказательство сходимости решения возмущённой задачи к решению усреднённой в норме пространства $W_2^1$ равномерно по $L_2$-норме правой части уравнения и получение такой же равномерной оценки скорости сходимости. Кроме того,  получены и аналогичные оценки разности решений в $L_2$-норме, причём за счёт перехода к более слабой норме удаётся улучшить оценку скорости сходимости.

Отметим, что нам известна лишь одна работа, где были установлены равномерные оценки, аналогичные нашим \cite{33}. В этой работе рассматривалась ограниченная трёхмерная область, строго периодически перфорированная вдоль плоскости. На границе областей задавалось классическое линейное третье краевое условие. Для различных случаев соотношений между размерами полостей, расстояний между ними и коэффициента в третьем краевом условии были получены равномерные по правой части оценки разности решений возмущённой и усреднённой задач. Подчеркнём, что рассматриваемый нами случай существенно более сложный ввиду произвольной непериодической структуры чередования и также произвольной размерности. При этом следует отметить, что размерность пространства является важным фактором, так как в размерности два и три
имеются теоремы о вложении пространства $W_2^2$ в пространство непрерывных функций и это облегчает  доказательство оценок в непериодическом случае, см. \cite{24}. В случае же произвольной размерности приходится проводить более тонкий анализ, см. определение нормы $\|\,\cdot\,\|_S$ в следующем параграфе и лемму~\ref{lm3.0} из третьего параграфа. Для строго периодического чередования техника доказательства оценок существенно упрощается и в произвольной размерности.

Опишем структуру статьи. В следующем параграфе описывается постановка задачи и формулируются основные результаты. В третьем параграфе мы обсуждаем различные случаи структуры перфораций, для которых справедливы наши основные результаты. В четвёртом параграфе мы приводим серию  вспомогательных лемм, которые далее используются в трёх последующих параграфах для доказательства основных результатов.

\section{Постановка задачи и формулировка результатов}
\label{s2}

Пусть $x=(x',x_n)$ и $x'=(x_1,\ldots,x_{n-1})$ -- декартовы координаты в $\mathds{R}^n$ и $\mathds{R}^{n-1}$ соответственно, $n\geqslant3$. Через $\Om$ обозначим произвольную область в $\mathds{R}^n$ с границей класса $C^2$. Пусть $S\subset\Om$ -- многообразие без края класса $C^3$ коразмерности $1$, которое либо замкнуто, либо бесконечно.

Обозначим через $\e$ малый положительный параметр, а через $\eta=\eta(\e)$ функцию, удовлетворяющую неравенству: $0<\eta(\e)\leqslant1$. Пусть $\mathbb{M}^\e\subseteq\mathds{N}$ -- произвольное множество. Выберем в окрестности многообразия $S$ точки $M_k^\e$, $k\in\mathbb{M}^\e$ так, чтобы выполнялось условие
$\dist(M_k^\e,S)\leqslant R_0\e$, где $R_0>0$ -- некоторая константа, не зависящая от $k$ и $\e$. Через $\omega_{k,\e}$, $k\in\mathbb{M}^\e$, обозначим ограниченные области в $\mathds{R}^n$ c границами класса $C^2$; допускается зависимость областей от $\e$. Положим:
\begin{equation*}
\om_k^\e:=\big\{x:\, (x-M_k^\e)\e^{-1}\eta^{-1}(\e)\in \om_{k,\e}\big\},\qquad \tht^\e:=\bigcup\limits_{k\in\mathbb{M}^\e}\om_k^\e.
\end{equation*}
Из области $\Om$ вырежем полости $\om_k^\e$, $k\in\mathbb{M}^\e$ и обозначим полученную область через $\Om^\e$, т.е., $\Om^\e:=\Om\setminus\tht^\e$, см. рис. 1.

\begin{figure}[t]
\begin{center}
\includegraphics[scale=0.45]{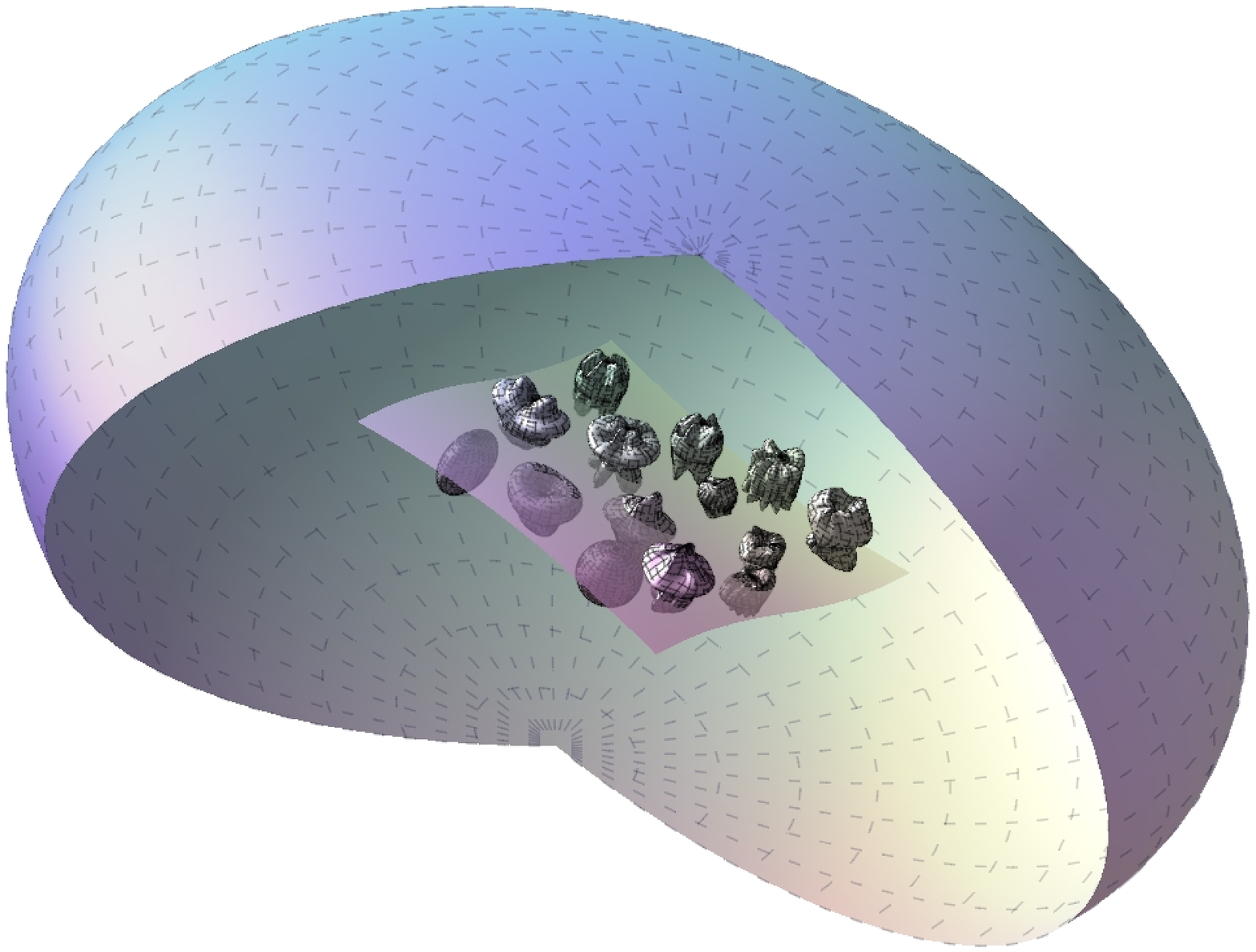}
\caption{Пример области, перфорированной вдоль многообразия}
\end{center}
\end{figure}

В области $\Om$ зададим функции $A_{ij}=A_{ij}(x)$, $A_j=A_j(x)$, $A_0=A_0(x)$, удовлетворяющие  условиям:
\begin{equation}\label{2.1}
\begin{aligned}
&A_{ij}\in W_\infty^1(\Om),\qquad  A_j, A_0\in L_\infty(\Om),\qquad A_{ij}=A_{ji},\qquad i,j=1,\ldots,n,
\\
&\sum\limits_{i,j=1}^n A_{ij}(x)z_i \overline{z_j}\geqslant c_0|z|^2,\qquad x\in\Om,\qquad z=(z_1\ldots,z_n)\in\mathds{C}^n,
\end{aligned}
\end{equation}
где $c_0>0$ -- некоторая константа, не зависящая от $x$ и $z$. Функции $A_{ij}$
являются вещественнозначными, а функции $A_j$, $A_0$ -- комплекснозначными. Через $a=a(x,u)$ обозначим некоторую комплекснозначную функцию, заданную для $u\in\mathds{C}$ и $x\in  \{x:\, \dist(x,S)\leqslant\tau_0\}$, где $\tau_0>0$ -- некоторое фиксированное число. Будем считать, что функция $a$ удовлетворяет следующим условиям:
\begin{equation}\label{2.2}
\begin{aligned}
&\left|\frac{\p a}{\p \RE u}(x,u)\right|+\left|\frac{\p a}{\p \IM u}(x,u)\right|\leqslant a_0,
\\
&\hphantom{\bigg|}a(u,0)=0, \qquad |\nabla_x a(x,u)|\leqslant a_1 |u|,
\end{aligned}
\end{equation}
где $a_0$ и $a_1$ -- некоторые константы, не зависящие от $x$ и $u$.
Пусть $f\in L_2(\Om)$ -- некоторая функция, $\lambda$ -- вещественное число.

В работе рассматривается следующая краевая задача:
\begin{equation}\label{2.5}
\begin{gathered}
\bigg(-\sum\limits_{i,j=1}^n \frac{\p}{\p x_i}A_{ij}\frac{\p}{\p x_j}+\sum\limits_{j=1}^n
A_j\frac{\p}{\p x_j}
+A_0-\lambda\bigg) u_\e=f\quad\text{в}\quad\Om^\e,
\\
 u_\e=0\quad\text{на}\quad\p\Om,\qquad\frac{\p u_\e}{\p\mathrm{n}}+a(\,\cdot\,,u_\e)=0\quad\text{на}\quad\p\tht^\e,
\end{gathered}
\end{equation}
где производная по конормали задана соотношением:
\begin{equation*}%\l%abel{2.3}
\frac{\p}{\p\mathrm{n}}=\sum\limits_{i,j=1}^n A_{ij}\nu_i\frac{\p}{\p x_j},%+\sum\limits_{j=1}^n\overline{A_j}\cos(\nu,Ox_j),
\end{equation*}
$\nu_i$ -- $i$-ая компонента единичной нормали $\nu$ к $\p\tht^\e$, направленная внутрь множества $\tht^\e$. Основной целью данной работы является исследование асимптотического поведения решения краевой задачи (\ref{2.5}) при $\e\to0$.

Основные результаты работы получены при выполнении некоторых условий на геометрию перфорации. Сформулируем эти условия. Через $\tau$ обозначим расстояние от точки до $S$, измеренное вдоль нормали, а через $s$ -- какие-нибудь локальные переменные на поверхности $S$. Наше первое условие означает определённую регулярность поверхности $S$.

\begin{enumerate}
\def\theenumi{{A\arabic{enumi}}}
\item\label{A1}
Существует
фиксированная константа  $c_1>0$, такая что переменные $(\tau,s)$ корректно определены по крайней мере в области $\{x:\,  \dist(x,S)\leqslant\tau_0\}$ и верны равномерные оценки:
\begin{equation*}
\left|\nabla_{(\tau,s)} x_i \right|\leqslant c_1,\qquad i=1,\ldots,n.
\end{equation*}
\end{enumerate}

Пусть $B_r(M)$ -- шар в $\mathds{R}^n$ с центром в точке $M$ радиуса $r$. На размеры и взаимное расположение полостей $\om_k^\e$ наложим следующее условие.
\begin{enumerate}
\def\theenumi{{A\arabic{enumi}}}
\setcounter{enumi}{1}
\item\label{A2}
Существуют точки $M_{k,\e}\in\om_k^\e$, $k\in\mathbb{M}^\e$ и числа $0<R_1<R_2$, $b>1$, не зависящие от $\e$, такие  что для достаточно малых $\e$ выполнено:
\begin{align*}
&B_{R_1}(M_{k,\e})\subset\om_{k,\e}\subset B_{R_2}(0),\qquad\hphantom{..} k\in \mathbb{M}^\e,
\\
&B_{bR_2\e}(M_k^\e)\cap B_{bR_2\e}(M_i^\e)=\emptyset, \qquad i,k\in\mathbb{M}^\e,\quad i\neq k.
\end{align*}
Для всех $k$ и $\e$ множества $B_{R_2}(0)\setminus\om_{k,\e}$ связны.
\end{enumerate}

В окрестности границ областей $\om_{k,\e}$ введём локальную переменную $\rho$ -- расстояние от точки до границы $\p\om_{k,\e}$, измеренное в направлении внешней нормали. Следующие два условия касаются форм областей $\p\om_{k,\e}$.
 \begin{enumerate}
 \def\theenumi{{A\arabic{enumi}}}
\setcounter{enumi}{2}
\item\label{A3}
Существуют фиксированные константы $\rho_0>0$, $c_2>0$ такие, что переменная
$\rho$ корректно определена по крайней мере на множествах $\{x:\, \dist(x,\p\om_{k,\e})\leqslant\rho_0\}\setminus\om_{k,\e}\subseteq B_{b_* R_2}(0)$, $b_*:=(b+1)/2$, одновременно для всех $k\in\mathbb{M}^\e$ и на данных множествах верны равномерные по $\e$, $\e$, $x$ и $k\in\mathbb{M}^\e$ оценки:
\begin{equation*}
\left|\frac{\p x_i}{\p\rho}\right|\leqslant c_2,\qquad i=1,\ldots,n.
\end{equation*}
 \end{enumerate}

 \begin{enumerate}
 \def\theenumi{{A\arabic{enumi}}}
\setcounter{enumi}{3}
\item\label{A4}
Существует обобщенное решение $X_k\in L_\infty(B_{b_*R_2}(0)\setminus\om_{k,\e})$, $k\in\mathbb{M}^\e$, краевой задачи:
\begin{equation}\label{2.6} 
\begin{gathered}
\dvr X_k=f_k\quad\text{в}\quad B_{b_* R_2}(0)\setminus\om_{k,\e},
\\
X_k\cdot\vartheta=-1\quad\text{на}\quad\p\om_{k,\e}, \qquad
X_k\cdot\vartheta=\phi_k\quad\text{на}\quad\p\om_{k,\e},
\end{gathered}
\end{equation}
где $\vartheta$ -- внешняя нормаль к $\p B_{b_*R_2}(0)$ и $\p\om_{k,\e}$,  $f_k\in L_\infty(B_{b_*R_2}(0)\setminus\om_{k,\e})$, $\phi_k\in L_2(\p B_{b_* R_2}(0))$, $k\in\mathbb{M}^\e$ -- некоторые функции, причём выполнено условие:
\begin{align}
&\int\limits_{B_{b_* R_2}(0)\setminus\om_{k,\e}} f_k \,dx=0. \label{2.7}
\end{align}

Функции $X_k$ и $f_k$ ограничены  в норме $L_\infty(B_{b_*R_2}(0)\setminus\om_{k,\e})$, равномерно по всем $k\in\mathbb{M}^\e$ и $\e$, а функции $\phi_k$ аналогично равномерно ограничены в норме $L_\infty\big(\p B_{b_* R_2}(0)\big)$.
\end{enumerate}

Отметим, что под обобщённым решением задачи (\ref{2.6}) для некоторой функции $f_k\in L_2(B_{b_*R_2(0)\setminus\om_{k,\e}})$ мы понимаем функцию $X\in L_2(B_{b_* R_2}(0)\setminus\om_{k,\e})$ такую, что
\begin{equation}\label{12.85}
\int\limits_{B_{b_*R_2(0)\setminus\om_{k,\e}}} f_k\overline{\psi} \,dx= \int\limits_{\p B_{b_*R_2}(0)}\phi_k  \overline{\psi}\,ds - \int\limits_{\p \om_{k,\e}} \overline{\psi}   \,ds
\end{equation}
для произвольной пробной функции $\psi\in C^1(\overline{B_{b_*R_2(0)}\setminus\om_{k,\e}})$, где функция $\phi_k$ -- из $L_2(\p B_{b_* R_2}(0))$. Условие~\ref{A4} дополнительно требует попадания функций $f_k$, $X_k$, $\phi_k$ в соответствующие $L_\infty$-пространства и равномерную ограниченность в нормах этих пространств.  

Пусть $\z=\z(t)$, $t\in[0,1]$ -- бесконечно дифференцируемая срезающая функция, принимающая значения из отрезка $[0,1]$, равная нулю при $|t|>1$ и удовлетворяющая условию
\begin{equation}\label{2.8b}
 \int\limits_{\mathbb{S}^{n-2}} \z(|t|)\,dt=1,
\end{equation}
где $\mathbb{S}^{n-2}$ --  единичная сфера в пространстве $\mathds{R}^{n-1}$.
Через $M_{k,\bot}^\e$ обозначим проекции точек $M_k^\e$ на поверхность $S$. На поверхности $S$ определим  функцию:
\begin{equation}\label{2.8}
\a^\e(x)=\left\{
\begin{aligned}
\frac{\eta^{n-1}|\p\om_{k,\e}|}{R_2^{n-1}}
\z&\left(\frac{|
x-M_{k,\bot}^\e|}{\e  R_2}\right)
&&\text{при}\quad |
x-M_{k,\bot}^\e|<\e R_2,\quad k\in \mathbb{M}^\e,\\
& 0 \quad &&\text{в остальных точках} \  S.
\end{aligned}
\right.
\end{equation}

Обозначим: $\vw:=\big\{x\in\mathds{R}^n:\, 0<\tau<\frac{\tau_0}{2}\big\}$.
Пусть
$\Phi$ -- произвольная функция, заданная на $S$ и являющаяся следом некоторой функции из $W_2^1(\vw)$, то есть, $\Phi\in W_2^{\frac{1}{2}}(S)$. Ясно, что следующие две задачи однозначно разрешимы в $W_2^1(\vw)$:
\begin{gather}\label{2.23}
\begin{gathered}
-\D U_\Phi^N + U_\Phi^N=0\quad\text{в}\quad \vw,
\\
\frac{\p U_\Phi^N}{\p\tau}=-\Phi\quad\text{на}\quad S,
\qquad\frac{\p U_\Phi^N}{\p\nu}=0\quad\text{на}\quad \p\vw\setminus S,
\end{gathered}
\\
\label{2.21}
\begin{gathered}
-\D U_\Phi^D + U_\Phi^D=0\quad\text{в}\quad \vw,
\\
 U_\Phi^D=\Phi\quad\text{на}\quad S,  \hphantom{\p-,}
\qquad\frac{\p U_\Phi^D}{\p\nu}=0\quad\text{на}\quad \p\vw\setminus S,
\end{gathered}
\end{gather}
где $\nu$ -- единичная нормаль к  поверхности $\p\vw\setminus S$, внешняя к области $\vw$. Далее в третьем параграфе будет показано (см. лемму~\ref{lm12.7}), что  следующая норма определена корректно по крайней мере на пространстве $L_\infty(S)$:
\begin{equation}\label{2.22}
\|\a\|_S^2:= \sup\limits_{ \substack{\Phi\in W_2^{\frac{1}{2}}(S)
\\ \Phi\ne 0}}\frac{\|U_{\a\Phi}^N\|_{W_2^1(\vw)}^2}{\|U_\Phi^D\|_{W_2^1(\vw)}^2},
\end{equation}
где $\a$ -- произвольная функция из $L_\infty(S)$.

На функцию $\a^\e$ наложим следующее условие.
\begin{enumerate}
\def\theenumi{{A\arabic{enumi}}}
\setcounter{enumi}{4}
\item\label{A5}
Существуют ограниченная измеримая функция $\a^0$, заданная на $S$ и принадлежащая $W_\infty^1(S)$,    и функция $\kappa=\kappa(\e)\to+0$ при $\e\to+0$ такие, что для всех достаточно малых $\e$ верны оценки:
\begin{equation*}%\l%abel{2.9}
\|\a^\e-\a^0\|_{S}\leqslant\kappa(\e).
\end{equation*}
\end{enumerate}

Обозначим через $\Ho^1(\Om^\e,\p\Om)$ подпространство функций из $W_2^1(\Om^\e)$, обращающихся в нуль на $\p\Om$. Решение краевой задачи (\ref{2.5}) будем понимать в обобщенном смысле. Обобщенным решением задачи (\ref{2.5}) называется функция $u_\e$, принадлежащая пространству $W_2^1(\Om^\e)$ и удовлетворяющая интегральному тождеству:
\begin{equation*}
\mathfrak{h}_a(u_\e,v)-\l(u,v)_{L_2(\Om^\e)}=(f,v)_{L_2(\Om^\e)}
\end{equation*}
для любых $v_\e\in\Ho^1(\Om^\e,\p\Om)$, где
\begin{align}
&
\mathfrak{h}_a(u,v):=\mathfrak{h}_0(u,v)+(a(\,\cdot\,,u),v)_{L_2(\p\tht^\e)},
\nonumber
\\
&
\begin{aligned}
\mathfrak{h}_0(u,v):=&\sum\limits_{i,j=1}^n\left(A_{ij}\frac{\p u}{\p x_j},\frac{\p v}{\p x_i}\right)_{L_2(\Om^\e)}+\sum\limits_{j=1}^n\left( A_j\frac{\p
u}{\p x_j},v\right)_{L_2(\Om^\e)}
\\
&
+(A_0 u,v)_{L_2(\Om^\e)}.
\end{aligned}\label{2.11}
\end{align}
Здесь интеграл по границе полостей $\p\tht^\e$ понимается в смысле следов. Далее мы докажем, что
условия \ref{A1}, \ref{A2}, \ref{A3} обеспечивают существование такого следа в пространстве $L_2(\p\tht^\e)$ (лемма \ref{lm3.1}). Также докажем, что при подходящем выборе параметра $\l$ задача (\ref{2.5}) имеет единственное решение (лемма \ref{lm3.2}).

Если выполнены условия ~\ref{A1},~\ref{A2},~\ref{A3},~\ref{A4} и одно из следующих условий: $a\equiv0$ или $\eta(\e)\to 0$, $\e\to0$, то при усреднении полости пропадают вместе с многообразием $S$ и усреднённая задача для (\ref{2.5}) имеет вид:
\begin{equation}\label{2.12}
\begin{aligned}
\bigg(-\sum\limits_{i,j=1}^n\frac{\p}{\p x_i} A_{ij}\frac{\p }{\p x_j}+\sum\limits_{j=1}^n
&A_j\frac{\p}{\p x_j}
+A_0-\lambda \bigg) u_0=f\quad\text{в}\quad
\Om,
\\
&u_0=0\quad\text{на}\quad\p\Om.
\end{aligned}
\end{equation}

Если же $\eta$ не стремится к нулю, а функция $a$ произвольна, то при выполнении условий~\ref{A1},~\ref{A2},~\ref{A3},~\ref{A4},~\ref{A5} усреднённая задача для  (\ref{2.5}) имеет вид
\begin{gather}
\bigg(-\sum\limits_{i,j=1}^n\frac{\p}{\p x_i} A_{ij}\frac{\p }{\p x_j}+\sum\limits_{j=1}^n
A_j\frac{\p}{\p x_j}
+A_0-\lambda \bigg)u_0=f\quad\text{в}\quad
\Om,\label{2.13}
\\
u_0=0\quad \text{на}\quad\p\Om,\label{2.15}
\qquad
[u_0]_S=0,\qquad\left[\frac{\p u_0}{\p \mathrm{n}}\right]_S+\a^0 a(\,\cdot\,, u_0)\big|_S=0,
\end{gather}
где  $[u]_S:=u|_{\tau=+0}-u|_{\tau=-0}$ -- скачок функции $u$ на $S$. В этом случае на многообразии $S$ возникает граничное условие из (\ref{2.15}). Отметим, что граничное условие (\ref{2.15}) описывает нелинейное дельта-взаимодействие на поверхности $S$. Решения задач (\ref{2.12}) и (\ref{2.13}),
(\ref{2.15}) также будем понимать в обобщенном смысле.

Наши  основные результаты сформулированы в следующих двух теоремах. Первая из них описывает ситуацию, когда при усреднении возникает задача (\ref{2.12}).

\begin{theorem}\label{th1}
Пусть выполнены предположения~\ref{A1},~\ref{A2},~\ref{A3},~\ref{A4}.
Тогда существует $\l_0$, не зависящее от $\e$, такое что при $\l<\l_0$ задачи (\ref{2.5}) и (\ref{2.12}) однозначно разрешимы для всех $f\in L_2(\Om)$. Если дополнительно выполнено одно из условий
\begin{equation}\label{2.16}
a\equiv0\quad\text{или}\quad\eta(\e)\to 0,\quad\e\to0,
\end{equation}
то справедливы неравенства:
\begin{equation}\label{2.17}
\|u_\e-u_0\|_{W_2^1(\Om^\e)}\leqslant C\big(\e\eta+\e^{\frac{1}{2}}\eta^{\frac{n}{2}}(\e)\big)\|f\|_{L_2(\Om)}, \hphantom{(\e)}
\quad\text{}\quad
\end{equation}
если $a\equiv0$,
и
\begin{equation}\label{2.18}
\|u_\e-u_0\|_{W_2^1(\Om^\e)}\leqslant C\big(\e\eta(\e)
+\eta^{n-1}(\e)\big)\|f\|_{L_2(\Omega)},
\end{equation}
 если $\eta(\e)\to 0$, $\e\to0$,
где константы $C$ не зависят от $\e$ и $f$, но зависят от $\lambda$.
\end{theorem}

Во второй теореме описывается ситуация, когда усреднение приводит к задаче (\ref{2.13}), (\ref{2.15}).

\begin{theorem}\label{th2}
Пусть выполнены предположения~\ref{A1},~\ref{A2},~\ref{A3},~\ref{A4},~\ref{A5}. Тогда существует $\lambda_0$, не зависящее от $\e$, $\eta$ и $f$, такое что при $\lambda<\lambda_0$ задачи (\ref{2.5}) и (\ref{2.13}),
(\ref{2.15}) однозначно разрешимы для всех $f\in L_2(\Om)$ и имеет место неравенство:
\begin{equation}\label{2.19}
\|u_\e-u_0\|_{W_2^1(\Om^\e)}\leqslant C\big(\e^\frac{1}{2}
+\kappa(\e)\big)\|f\|_{L_2(\Om)},
\end{equation}
где константа $C$ не зависит от $\e$ и $f$, но зависит от $\lambda$.
\end{theorem}

В следующих двух теоремах мы показываем, что ослабляя норму для разности решений возмущённой и соответствующей усреднённой задачи, мы добиваемся более высокой скорости сходимости.

\begin{theorem}\label{th3}
Пусть $A_j\in W_\infty^1(\Om)$, выполнены предположения~\ref{A1},~\ref{A2}, \ref{A3},~\ref{A4} и одно из  условий в
(\ref{2.16}). Тогда для решений задач (\ref{2.5}), (\ref{2.12})
верны оценки
\begin{equation}\label{2.20}
\begin{aligned}
\|u_\e-u_0\|_{L_2(\Om^\e)}\leqslant & C\big(\e^2\eta^2(\e)+\e\eta^n(\e)\big)\|f\|_{L_2(\Om^\e)}
\\
&+C
\big(\e\eta(\e)+\e^{\frac{1}{2}}\eta^{\frac{n}{2}}(\e)\big) \|f\|_{L_2(\tht^\e)},
\end{aligned}
\end{equation}
если $a\equiv0$, и
\begin{equation}\label{2.4}
\begin{aligned}
\|u_\e-u_0\|_{L_2(\Om^\e)}\leqslant &  C(\e^2\eta(\e)+\eta^{n-1}(\e)) \|f\|_{L_2(\Omega^\e)}
\\
&+C
\big(\e\eta(\e)+\e^{\frac{1}{2}}\eta^{\frac{n}{2}}(\e)\big) \|f\|_{L_2(\tht^\e)},
\end{aligned}
\end{equation}
если $\eta(\e)\to 0$ при  $\e\to0$. В этих оценках  константы $C$ не зависят от $\e$ и $f$, но зависят от $\lambda$.
\end{theorem}

\begin{theorem}\label{th4}
Пусть $A_j\in W_\infty^1(\Om)$ и выполнены предположения~\ref{A1},~\ref{A2}, \ref{A3},~\ref{A4},~\ref{A5}. Тогда для решений задач (\ref{2.5}), (\ref{2.13}), (\ref{2.15}) верна оценка
\begin{equation}\label{2.9}
\|u_\e-u_0\|_{L_2(\Om^\e)}\leqslant     C(\e+\kappa(\e)) \|f\|_{L_2(\Omega^\e)}
+C
\big(\e\eta(\e)+\e^{\frac{1}{2}}\eta^{\frac{n}{2}}(\e)\big) \|f\|_{L_2(\tht^\e)},
\end{equation}
 где  константа $C$ не зависит от $\e$ и $f$, но зависит от $\lambda$.
\end{theorem}

Кратко обсудим полученные результаты. Уравнение в задаче (\ref{2.5}) является линейным эллиптическим уравнением второго порядка, при этом на границах полостей ставится третье нелинейное граничное условие. Отверстия распределены вдоль поверхности $S$, которая должна быть достаточно регулярной. Помимо предполагаемой гладкости, регулярность включает в себя условие~\ref{A1}, которое фактически исключает нарастающие осцилляции этой поверхности в случае, когда она бесконечна. Для компактных поверхностей условие~\ref{A1} автоматически вытекает из её гладкости.

Согласно условию~\ref{A2}, между полостями имеется минимальное расстояние порядка $O(\e)$, которое гарантирует непересечение соседних полостей. Подчеркнём, что речь идет исключительно о минимальном расстоянии и не предполагается одновременное наличие и верхней оценки порядка $O(\e)$. В частности, допускается ситуация, когда расстояния между какими-то соседними полостями будут много больше $\e$. Линейные размеры всех полостей порядка $O(\e\eta(\e))$, что также гарантируется условием~\ref{A2}. Параметр $\eta$ при этом   описывает отношение между характерными размерами полостей и расстояния между ними.

Форма границ полостей и их распределение вдоль многообразия произвольные. Никаких существенных условий на структуру чередования не налагается. Помимо естественных ограничений в условии~\ref{A2}, также налагаются условия~\ref{A3} и~\ref{A4}. Оба условия означают определенную регулярность границ полостей; вопрос о том, возможно ли одно условие вывести из другого или заменить их на единое более простое условие, остался открытым.

При выполнении условий ~\ref{A1},~\ref{A2},~\ref{A3},~\ref{A4} и одного из условий (\ref{2.16}) усреднённая задача для (\ref{2.5}) имеет вид (\ref{2.12}).
В этом случае полости пропадают вместе с многобразием $S$, вдоль которого они расположены и усреднённая задача (\ref{2.12}) никак не зависит от выбора многообразия $S$. При выполнении условий~\ref{A1},~\ref{A2},~\ref{A3},~\ref{A4} и дополнительного условия~\ref{A5} усреднённая задача для (\ref{2.5}) имеет вид (\ref{2.13}),
(\ref{2.15}). Теперь усреднённая задача зависит от выбора многообразия $S$, на котором возникает граничное условие, которое уместно трактовать как нелинейное дельта-взаимодействие.  Коэффициент в этом условии определяется геометрией и распределением полостей. А именно, функция $\a^\e$  зависит от распределения проекций точек $M_{k,\bot}^\e$ на поверхности $S$ и от площадей границ полостей $\p\om_{k,\e}$. При малых $\e$ эта функция должна оказываться близкой к некоторой функции $\a$ в смысле нормы $\|\,\cdot\,\|_S$, то есть, в условии~\ref{A5} речь идёт об усреднении функции $\a^\e$ в смысле  нормы $\|\,\cdot\,\|_S$ и это налагает определенные ограничения на степень непериодичности распределения точек $M_k^\e$ и произвол в выборе полостей $\om_{k,\e}$. При этом следует подчеркнуть, что форма полостей   оказывается неважной, а роль играют лишь площади поверхностей их границ, так как именно они  входят в определение функции $\a^\e$. Примеры возможных непериодических распределений и соответствующие им функции $\a^\e$ и $\a^0$ мы обсудим в следующем отдельном параграфе, сейчас же лишь отметим, что норму $\|\,\cdot\,\|_S$ можно рассматривать как норму мультипликатора из пространства $W_2^{\frac{1}{2}}(S)$ в
$W_2^{-\frac{1}{2}}(S)$.

Теорема \ref{th1} утверждает сходимость решения задачи (\ref{2.5}) к решению задачи (\ref{2.12}) в $W_2^1$ равномерно по правой части уравнения. Теорема \ref{th2} утверждает аналогичную сходимость решения задачи  (\ref{2.5}) к решению задачи  (\ref{2.13}),
(\ref{2.15}). Помимо сходимости, теоремы \ref{th1} и \ref{th2} дают оценки скорости сходимости, см. неравенства (\ref{2.17}), (\ref{2.18}), (\ref{2.19}). В частном случае, когда краевое условие на границах полостей является линейным, утверждения теорем \ref{th1} и \ref{th2} означают наличие равномерной резольвентной сходимости в смысле нормы операторов, действующих из $L_2$ в $W_2^1$, и дают оценки скорости сходимости в смысле операторной нормы. С этой точки зрения наши основные результаты оказываются того же характера, что и известные результаты об операторных оценках в линейных задачах граничного усреднения \cite{15},  \cite{16}, \cite{17},  \cite{18},  \cite{19},  \cite{20},  \cite{21},  \cite{22},  \cite{23},  \cite{24}. По сравнению с цитированными работами, теоремы~\ref{th3},~\ref{th4} даёт качественно новый результат об оценке разности решений в $L_2$-норме. Оценки в этих теоремах устанавливают более высокую скорость за счёт ослабления нормы. Здесь мы имеем ввиду первые слагаемые в правых частях оценок (\ref{2.20}), (\ref{2.4}), (\ref{2.9}). Вторые слагаемые имеют тот же порядок малости, что в оценках из теорем~\ref{th1},~\ref{th2}. Однако следует подчеркнуть, что данные вторые слагаемые содержат нормы $\|f\|_{L_2(\tht^\e)}$, которые определяются значениями функции $f$ внутри отверстий -- эта функция исходно задаётся сразу на всей области $\Om$ лишь для упрощения формулировки усреднённой задачи. Вместе с тем, эти значения \textit{не участвуют} в исходной задаче (\ref{2.5}), так как она ставится в перфорированной области. В частности, можно зафиксировать достаточно малое $\e$ и выбрать произвольную функцию $f\in L_2(\Om^\e)$, а затем \textit{продолжить её нулём внутрь отверстий} $\tht^\e$. Тогда для данного значения $\e$ будут верны все наши четыре основные теоремы, причём в оценках теорем~\ref{th3},~\ref{th4} вторые слагаемые в правых частях в этом случае пропадут.

Для доказательства теорем~\ref{th3},~\ref{th4} мы используем подход, изначально предложенный в работах \cite{Sen1}, \cite{Sen2}, см. также \cite{Pas1}, \cite{Pas2}. Подчеркнём вместе с тем, что техническая реализация этого подхода в нашем случае   отличается от цитированных работ, что связано с формальной несамосопряжённостью дифференциальных выражений в уравнениях в (\ref{2.5}), (\ref{2.12}), (\ref{2.13}), а также с нелинейностью краевых условий в (\ref{2.5}), (\ref{2.15}).

\section{Примеры перфораций}

В настоящем параграфе мы обсуждаем норму $\|\cdot\|_S$, определенную в (\ref{2.22}),  условие~\ref{A5} и примеры выбора форм и распределений полостей $\om_{k,\e}$, которые обеспечивают выполнение данного условия.

\subsection{Корректная определённость нормы $\|\,\cdot\,\|_S$}

В настоящем разделе мы доказываем, что соотношение  (\ref{2.22})
корректно определяет норму $\|\,\cdot\,\|_S$.

\begin{lemma}\label{lm3.0}
Формула (\ref{2.22}) определяет норму в пространстве $L_\infty(S)$. Для произвольной функции $\Phi\in W_2^{\frac{1}{2}}(S)$ верны равенство и оценки
\begin{gather}\label{12.32}
(\a\Phi,U_{\a\Phi}^N)_{L_2(S)}=\|U_{\a\Phi}^N\|_{W_2^1(\vw)}^2,
\\
\label{12.33}
\|\a\|_S\leqslant C\|\a\|_{L_\infty(S)},\qquad
\big|(\a u,v)_{L_2(S)}\big|\leqslant  \|\a\|_S \|u\|_{W_2^1(\vw)} \|v\|_{W_2^1(\vw)},
\end{gather}
где  $u$, $v$ -- произвольные функции из $W_2^1(\vw)$, а $C$ -- некоторая константа, не зависящая от $\a$.
\end{lemma}

\begin{proof}
Для проверки равенства (\ref{12.32}) достаточно выписать определение обобщенного решения задачи (\ref{2.23}), взяв $U_{\a\Phi}^N$ в качестве пробной функции.

Докажем, что правая часть в (\ref{2.22}) определена корректно и является нормой. Из равенства (\ref{12.32}) и стандартных теорем об оценке следа функции следует, что
\begin{align*}
\|U_{\a\Phi}^N\|_{W_2^1(\vw)}^2\leqslant &\|\a\Phi\|_{L_2(S)} \|U_{\a\Phi}^N\|_{L_2(S)} \leqslant C\|\a\|_{L_\infty(S)} \|\Phi\|_{L_2(S)} \|U_{\a\Phi}^N\|_{W_2^1(\vw)}
\\
\leqslant & C\|\a\|_{L_\infty(S)} \|U_\Phi^D\|_{W_2^1(\vw)} \|U_{\a\Phi}^N\|_{W_2^1(\vw)},
\end{align*}
где $C$ -- некоторые константы, не зависящие от $\Phi$, $\a$, $U_\Phi^D$, $U_{\a\Phi}^N$. Из полученной оценки вытекает, что отношение в правой части (\ref{2.22}) ограничено величиной $C\|\a\|_{L_\infty(S)}^2$ и потому супремум в (\ref{2.22}) существует. Кроме того, верна первая оценка в (\ref{12.33}).

Очевидно, что $\|\a\|_S=0$ если и только если $\a=0$. Однородность нормы и неравенство треугольника легко выводятся из очевидных равенств
$U_{C\a\Phi}=C U_{\a\Phi}$ и $U_{(\a_1+\a_2)\Phi}=U_{\a_1\Phi}+U_{\a_2\Phi}$. Поэтому формула (\ref{2.22}) действительно определяет норму.

Докажем теперь вторую оценку в (\ref{12.33}). Пусть $U_v^D$ -- решение задачи (\ref{2.11}), где в качестве правой части краевого условия на $S$ взят след функции $v$ на $S$. Тогда из определения обобщенного решения задачи (\ref{2.21}) следует, что
\begin{equation*}
(U_v^D,U_v^D-v)_{W_2^1(\vw)}=0,\qquad \|U_v^D\|_{W_2^1(\vw)}^2=(U_v^D,v)_{W_2^1(\vw)}.
\end{equation*}
Используя эти равенства, выводим:
\begin{equation*}
0\leqslant  \|v-U_v^D\|_{W_2^1(\vw)}^2=(v,v-U_v^D)_{W_2^1(\vw)},
\end{equation*}
а потому
\begin{equation}\label{12.34}
\|v\|_{W_2^1(\vw)}^2\geqslant (v,U_v^D)_{W_2^1(\vw)}=\|U_v^D\|_{W_2^1(\vw)}^2.
\end{equation}

Из определения обобщённого решения задачи (\ref{2.23}) с пробной функцией $U_v^D$ следует равенство $(\a u,v)_{L_2(S)}=(U_{\a u}^N,U_v^D)_{W_2^1(\vw)}$. Поэтому в силу неравенства Коши-Буняковского, равенства (\ref{12.32}), оценки (\ref{12.34}) и определения нормы $\|\a\|_S$ получаем:
\begin{equation*}
\frac{\big|(\a u,v)_{L_2(S)}\big|}{\|v\|_{W_2^1(\vw)}\|u\|_{W_2^1(\vw)}} \leqslant \frac{\big|(U_{\a u}^N, U_v^D)_{W_2^1(\vw)}\big|}{\|U_v^D\|_{W_2^1(\vw)}\|u\|_{W_2^1(\vw)}} \leqslant \frac{\|U_{\a u}^N\|_{W_2^1(\vw)}}{\|u\|_{W_2^1(\vw)}}\leqslant \|\a\|_S.
\end{equation*}
 Отсюда уже вытекает вторая оценка в (\ref{12.33}). Лемма доказана.
\end{proof}

Подчеркнём, что данная лемма не утверждает, что пространство $L_\infty(S)$ полное относительно нормы $\|\cdot\|_S$.
Также отметим, что данная норма по сути является нормой мультипликатора из пространства $W_2^{\frac{1}{2}}(S)$ в
$W_2^{-\frac{1}{2}}(S)$.

\begin{lemma}\label{lm12.9}
Пусть выполнены условия~\ref{A1},~\ref{A2},~\ref{A3},~\ref{A4}. Тогда
площади $|\p\om_{k,\e}|$ ограничены равномерно по $\e$ и $k$.
\end{lemma}
\begin{proof}
В определении (\ref{12.85}) обобщённого решения задачи (\ref{2.6}) выберем $\psi\equiv 1$. Тогда с учётом условия (\ref{2.7}) 
%Используя задачу для $X_k$ из (\ref{2.6}) и условие (\ref{2.7}), %проинтегрируем по частям следующим образом:
%\begin{align*}
%0=&\int\limits_{B_{b_*R_2(0)\setminus\om_{k,\e}}} f_k %\,dx=\int\limits_{B_{b_*R_2(0)\setminus\om_{k,\e}}} \dvr X_k \,dx
%\\
%=&\int\limits_{\p B_{b_*R_2}(0)} X_k\cdot\vartheta \,ds+\int\limits_{\p %\om_{k,\e}} X_k\cdot\vartheta \,ds.
%\end{align*}
%Из последнего равенства в силу граничного условия краевой задачи %(\ref{2.6}) и свойства  (\ref{2.7}) следует, что функция
%\begin{equation}\label{12.27}
%\phi_k:=X_k\cdot\vartheta,
%\end{equation}
%заданная на $\p B_{b_*R_2}(0)$, является элементом пространства %$L_{\infty}(\p B_{b_*R_2}(0))$ и удовлетворяет условию:
получим: 
\begin{equation*}%\l%abel{3.9}
\int\limits_{\p B_{b_*R_2}(0)}\phi_k\,ds=|\p \om_{k,\e}|.
\end{equation*}
Так как по условию~\ref{A4} функция $\phi_k$ ограничена в норме $L_\infty(\p B_{b_*R_2}(0))$ равномерно по $k$ и $\e$, то из полученного равенства немедленно следует, что все величины $|\p \om_{k,\e}|$ ограничены равномерно по $k$ и $\e$. Лемма доказана.
\end{proof}

\begin{lemma}\label{lm12.7}
Пусть условие~\ref{A5} выполнено для некоторой перфорации, удовлетворяющей условиям~\ref{A1},~\ref{A2},~\ref{A3},~\ref{A4}. Тогда для любой другой перфорации, удовлетворяющей тем же условиям и описываемой точками $\tilde{M}_k^\e$, $k\in\mathbb{M}^\e$ и полостями $\tilde{\om}_{k,\e}$, $k\in\mathbb{M}^\e$ такими, что выполнена равномерная по $k$ и $\e$ оценка
\begin{equation*}%\l%abel{12.35}
\e^{-1}\big|\tilde{M}_k^\e-M_k^\e\big| +\big||\p\om_{k,\e}|-|\p\tilde{\om}_{k,\e}|\big| \leqslant \mu(\e),
\end{equation*}
где $\mu(\e)$ -- некоторая функция, бесконечно малая при $\e\to+0$,  условие~\ref{A5} выполнено с той же функцией $\a^0$ и с заменой $\kappa(\e)$ на $\kappa(\e)+C\mu(\e)\eta^{n-1}(\e)$, где $C$ -- некоторая константа, не зависящая от $\e$.
\end{lemma}

\begin{proof}
Пусть $\tilde{\a}_\e$ -- это функция, построенная по формуле (\ref{2.8}) для перфорации, описываемой точками $\tilde{M}_k^\e$ и полостями $\tilde{\om}_{k,\e}$. В силу леммы Адамара выполнено
\begin{equation*}
\left| \z\left(\frac{|
x-M_{k,\bot}^\e|}{\e  R_2}\right) - \z\left(\frac{|
x-\tilde{M}_{k,\bot}^\e|}{\e  R_2}\right)
\right| \leqslant C\e^{-1}|M_k^\e-\tilde{M}_k^\e|,
\end{equation*}
где $C$ -- некоторая константа, не зависящая от $\e$ и $k$. Тогда из условия леммы сразу получаем оценку
\begin{equation*}
\|\a^\e-\tilde{\a}^\e\|_{L_\infty(S)}\leqslant C\mu(\e)\eta^{n-1}(\e),
\end{equation*}
где константа $C$ не зависит от $\e$. Учитывая теперь первую оценку в (\ref{12.33}),   легко выводим неравенство
\begin{equation*}
\|\tilde{\a}^\e-\a^0\|_S\leqslant \kappa(\e) + C\mu(\e) \eta^{n-1}(\e),
\end{equation*}
которое завершает доказательство леммы.
\end{proof}

Последняя лемма существенно расширяет класс перфораций, для которых выполнено условие~\ref{A5}. А именно, если это условие выполнено для какой-то перфорации, определяемой набором точек $M_k^\e$ и областей $\p\om_k^\e$, то оно выполнено с той же самой функцией $\a^0$ для перфораций, полученных произвольными малыми смещениями точек $M_k^\e$ и вариацией площадей $|\p\om_k^\e|$. Подчеркнём ещё, что форма полостей не играет никакой роли, а важна лишь площадь поверхности границы полости. Этот факт предоставляет большой произвол в выборе областей $\om_k^\e$.

\subsection{Примеры редко распредёленных перфораций}

В настоящем разделе мы обсуждаем два достаточно общих примера перфораций, для которых условие~\ref{A5} гарантированно выполняется с функцией $\a^0=0$.

Первый пример является прямым следствием лемм~\ref{lm3.0},~\ref{lm12.9}. А именно, пусть выполнены условия~\ref{A1},~\ref{A2},~\ref{A3},~\ref{A4}. Тогда
из определения функции $\a^\e$ и лемм~\ref{lm3.0},~\ref{lm12.9} немедленно вытекает равномерная по $\e$  оценка:
\begin{equation}\label{12.24}
\|\a^\e\|_S\leqslant \|\a^\e\|_{L_\infty(S)}\leqslant C\eta^{n-1}(\e),
\end{equation}
где $C$ -- некоторые константы, не зависящие от $\e$.  Следовательно,   если $\eta\to0$, то для любой перфорации, удовлетворяющей условиям~\ref{A1}, \ref{A2}, \ref{A3}, \ref{A4}, условие~\ref{A5} выполняется с $\a=0$. Отметим, что этот результат частично воспроизводит утверждение теоремы~\ref{th1} для случая $\eta(\e)\to0$.

Определим теперь покрытие поверхности $S$. Для этого выберем точки $\rL_p\in S$, $p\in\mathds{N}$ и фиксированное число $R_3>0$ такие, что \begin{equation}\label{12.41}
S\subset \bigcup\limits_{k\in\mathds{N}} B_{R_3}(\rL_p),\qquad \frac{6}{5}R_3\leqslant \inf\limits_{p\ne j} |\rL_p-\rL_j|\leqslant \frac{8}{5}R_3.
\end{equation}
Ясно, что такое покрытие всегда существует
с некоторым $R_3$. Также в силу неравенства в (\ref{12.41}) очевидно, что каждая точка поверхности $S$ попадает в конечное число шаров $B_{R_3}(\rL_k)$ и это число ограничено равномерно по всем точкам поверхности $S$.

Положим:
\begin{equation}\label{12.45}
N_\e:=\sup\limits_{p\in\mathds{N}} \#\big\{k:\, M_{k,\bot}^\e\in S\cap B_{R_3}(\rL_p)\big\},
\end{equation}
где символ $\#$ обозначает число элементов во множестве. Отметим, что данную величину можно интерпретировать как плотность распределения точек $M_k^\e$, так как она характеризует количество проекций $M_{k,\bot}^\e$ этих точек на каждом куске $S\cap B_{R_3}(\rL_p)$ поверхности.

Наш второй пример основан на следующей вспомогательной лемме.

\begin{lemma}\label{lm12.6}
Справедлива оценка
\begin{equation*}%\l%abel{12.28}
\|\a^\e\|_S\leqslant C\e\eta^{n-1} N_\e,
\end{equation*}
где $C$ -- некоторая константа, не зависящая от $\e$.
\end{lemma}

\begin{proof}
Произвольно фиксируем точку $\rL_p\in S$ и произвольно выберем точку $M_{k,\bot}^\e\in B_{R_3}(\rL_p)\cap S$. В окрестности поверхности $S$  введём локальные переменные $(s,\tau)$, где $s\in S$.  Обозначим:
\begin{align*}
&S_k^\e:=\big\{x\in S:\, |
x - M_{k,\bot}^\e|<\e R_2\big\},
 \\
 &\vw_p:=\left\{x\in\Om:\,   s\in B_{2R_3}\cap S,\ 0<\tau<\frac{\tau_0}{2}\right\}.
\end{align*}

Пусть $u\in W_2^1(\vw)$ -- произвольная функция. Ключевым шагом в доказательстве леммы является проверка следующей оценки:
\begin{equation}\label{12.42}
\|u\|_{L_2(S_k^\e)}\leqslant C\e^{\frac{1}{2}} \|u\|_{W_2^1(\vw_p)}.
\end{equation}
Здесь и всюду далее в доказательстве символом $C$ обозначаем различные несущественные константы, не зависящие от выбора функции $u$, параметров $\e$, $k$, $p$ и пространственных переменных.

Докажем оценку (\ref{12.42}). Функцию $u$ продолжим чётным образом по переменной $\tau$, а именно, положим $u(s,\tau):=u(s,-\tau)$. Продолжение очевидно оказывается элементом пространства $W_2^1(\vw^+)$, $\vw^+:=\{x:\, |\tau|<\frac{\tau_0}{2}\}$ и верна оценка
\begin{equation}\label{12.43}
\|u\|_{W_2^1(\vw^+_p)}\leqslant C\|u\|_{W_2^1(\vw_p)},\qquad \vw^+_p:=\left\{x:\, s\in B_{2R_3}\cap S,\ |\tau|<\frac{\tau_0}{2}\right\}.
\end{equation}

Пусть $\chi=\chi(t)$ -- бесконечно дифференцируемая срезающая функция, равная единице при $t<1$ и нулю при $t>2$. Для $s\in S_k^\e$ из очевидного равенства
\begin{equation*}
u(s)=\int\limits_{2\e}^{0} \frac{\p\ }{\p\tau} u(x)\chi \left(\frac{\tau}{\e}\right)\,d\tau
\end{equation*}
и неравенства Коши-Буняковского легко выводим, что
\begin{equation*}
|u(s)|^2\leqslant C \int\limits_{0}^{2\e} \left(\e \left|\frac{\p u}{\p\tau}(x)\right|^2 + \e^{-1}|u(x)|^2\right)\,d\tau.
\end{equation*}
Интегрируя эту оценку по $S_k^\e$, с учётом условия \ref{A1} получаем:
\begin{equation}\label{12.44}
\|u\|_{L_2(S_k^\e)}^2\leqslant C\left(\e\|\nabla u\|_{L_2(\vw_{k,\e})}^2 + \e^{-1}\|u\|_{L_2(\vw_{k,\e})}^2 \right),
\end{equation}
где
\begin{equation*}
\vw_{k,\e}:=\left\{x\in\Om:\,
s\in S_k^\e,\, |\tau|<\frac{\tau_0}{2}\right\}.
\end{equation*}
Применяя теперь оценку
\begin{equation*}
\|u\|_{L_2(\vw_{k,\e})}\leqslant C \e\|u\|_{W_2^1(\vw_p^+)},
\end{equation*}
вытекающую из леммы~2.1 в \cite{32}, из (\ref{12.44}) получаем неравенство (\ref{12.42}).

Используя теперь оценки (\ref{12.42}), (\ref{12.43}) и определение функции $\a^\e$, из равенства (\ref{12.32}) и свойств покрытия поверхности $S$ шарами $B_{2R_3}(\rL_p)$ выводим:
\begin{align*}
\|U_{\a^\e\Phi}^N\|_{W_2^1(\vw)}^2=&(\a^\e\Phi,U_{\a^\e\Phi}^N)_{L_2(S)}
\leqslant   C\sum\limits_{k\in\mathbb{M}^\e} \|\Phi\|_{L_2(S_k^\e)} \|U_{\a^\e\Phi}^N\|_{L_2(S_k^\e)}
\\
\leqslant & C\e \sum\limits_{k\in\mathbb{M}^\e}
\|U_{\Phi}^D\|_{W_2^1(\vw_p)}
\|U_{\a^\e \Phi}^N\|_{L_2(\vw_p)},
\end{align*}
где для каждого $k$ параметр $p$ выбран из условия $M_{k,\bot}^\e\in B_{R_3}(\rL_p)\cap S$. С учётом такого выбора $p$ и определения числа $N_\e$ в (\ref{12.45}), продолжим оценки:
\begin{align*}
\|U_{\a^\e\Phi}^N\|_{W_2^1(\vw)}^2
 \leqslant & C\e N_\e \sum\limits_{p\in\mathds{N}} \|U_{\Phi}^D\|_{W_2^1(\vw_p)}
\|U_{\a^\e \Phi}^N\|_{L_2(\vw_p)}
\\
 \leqslant & C\e N_\e   \|U_{\Phi}^D\|_{W_2^1(\vw^+)}
\|U_{\a^\e \Phi}^N\|_{L_2(\vw^+)}
 \\
\leqslant &  C\e N_\e   \|U_{\Phi}^D\|_{W_2^1(\vw)}
\|U_{\a^\e \Phi}^N\|_{L_2(\vw)}.
\end{align*}
Подставляя эту оценку в (\ref{2.22}), приходим к утверждению леммы.
\end{proof}

Из доказанной леммы следует, что если $\e N_\e\to +0$ при $\e\to+0$, то условие~\ref{A5} выполнено с $\a^0=0$. Описанное условие на $N_\e$ означает, что   плотность распределения точек $M_k^\e$ достаточно мала. Подчеркнём, что это условие не означает, что расстояния между точками $M_k^\e$ много больше, чем размеры  полостей. Такую ситуацию мы описываем с помощью параметра $\eta(\e)$, предполагая, что $\eta(\e)\to+0$ при $\e\to+0$. Лемма~\ref{lm12.6} в первую очередь ориентирована на ситуации, когда в окрестности отдельных частей поверхности $S$ точки $M_k^\e$ расположены друг от друга на расстояниях того же порядка малости, что и размеры полостей, но при этом их количество в окрестности кусков $S\cap B_{R_3}(\rL_p)$ мало. В качестве примера можно упомянуть ситуацию, когда точки $M_k^\e$ распределены небольшими кластерами: точки $M_k^\e$ расположены в окрестности кусков
поверхности линейного размера порядка $O(\e^{1-\b})$ с $\b<\frac{1}{d-1}$, а сами куски находятся друг от друга на расстоянии порядка $O(1)$.

\subsection{Периодические и локально-периодические перфорации}

Важным примером перфораций являются периодические и локально-{пе\-рио\-ди\-чес\-кие} перфорации. В свете имеющихся классических результатов о сильной и слабой сходимости решений задач в областях, перфорированных вдоль многообразий \cite{1}, \cite{2}, \cite{3}, \cite{4}, \cite{5}, \cite{6}, \cite{7}, необходимо гарантировать выполнение основных результатов о равномерной сходимости по крайней мере для периодических перфораций. Этому и посвящён настоящий раздел.

Начнём со вспомогательной леммы, которая далее будет играть ключевую роль в исследовании случаев периодических и локально-периодических перфораций.

\begin{lemma}\label{lm12.4}
Пусть существуют функции $\a^0\in W_\infty^1(S)$, $\Psi^\e\in W_\infty^2(\vw)$ такие, что
\begin{equation}\label{12.46}
\begin{aligned}
\|\Psi^\e\|_{L_\infty(\vw)} &+ \|\D \Psi^\e\|_{L_\infty(\vw)} +
\left\|\frac{\p\Psi^\e}{\p\nu}\right\|_{L_\infty(\p\vw\setminus S)}
\\
&+ \left\|\frac{\p\Psi^\e}{\p\tau}+\a^\e-\a^0\right\|_{L_\infty(S)}
=:\mu(\e)\to+0,\quad \e\to+0.
\end{aligned}
\end{equation}
Тогда существует константа $C$, не зависящая от $\e$ такая, что
\begin{equation}\label{12.47}
\|\a^\e-\a^0\|_S\leqslant C\mu^\frac{1}{2}(\e).
\end{equation}
\end{lemma}

\begin{proof}
Положим $\a:=\a^\e-\a^0$.
 Через $\cA_\a$ обозначим линейный оператор в $W_2^1(\vw)$, отображающий каждую функцию $u\in W_2^1(\vw)$ в решение задачи (\ref{2.23}) с $\Phi=\a u\big|_S$, где $\a\in L_\infty(S)$ -- некоторая вещественная функция. На основе равенства (\ref{12.32}) и второй оценки в (\ref{12.33}) несложно убедиться, что оператор $\cA_\a$ ограничен, самосопряжён и
\begin{equation}\label{12.83}
(\a u,U_{\a u}^N)_{L_2(S)}=(\cA_\a u, \cA_\a u)_{W_2^1(\vw)}=(\cA_\a^2 u,u)_{W_2^1(\vw)}.
\end{equation}
Из этого равенства, принципа минимакса, определения (\ref{2.12}) нормы $\|\a\|_S$ и неравенства (\ref{12.34}) следует, что величина $\|\a\|_S$ -- это верхняя точка спектра самосопряжённого оператора $\cA_\a^2$. Эта точка может быть точкой существенного спектра либо дискретным собственным значением. В обоих случаях существует последовательность функций $u_n\in W_2^1(\vw)$, $n\in\mathds{N}$, такая что
\begin{equation}\label{12.48}
\|u_n\|_{W_2^1(\vw)}=1,\qquad \|f_n\|_{W_2^1(\vw)}\to0,\quad n\to+\infty,
\end{equation}
где обозначено $f_n:=\big(\cA_\a^2-\|\a\|_S\big)u_n$.
Положим:
 \begin{equation*}
 v_n:=\cA_\a u_n,\qquad w_n:=\|\a\|_S u_n+f_n=\cA_\a v_n.
\end{equation*}
Из определений оператора $\cA_\a$ и обобщённого решения задачи (\ref{2.23}) вытекает справедливость интегральных тождеств
\begin{equation}\label{12.49}
(v_n,\vp)_{W_2^1(\vw)}=(\a u_n,\vp)_{L_2(S)}, \qquad
(w_n,\vp)_{W_2^1(\vw)}=(\a v_n,\vp)_{L_2(S)}
\end{equation}
для всех $\vp\in W_2^1(\vw)$. Из первого тождества с $\vp=v_n$, первого равенства в (\ref{12.48}) и второй оценки в (\ref{12.33})  элементарно вытекает неравенство
\begin{equation}\label{12.52}
\|v_n\|_{W_2^1(\vw)}\leqslant \|\a\|_{S}\|u_n\|_{W_2^1(\vw)}=\|\a\|_{S}.
\end{equation}

Отметим ещё, что из соотношений (\ref{12.48}) и равенств (\ref{12.83}) легко следует, что
\begin{equation}\label{12.84}
\begin{aligned}
&(\cA_\a u_n, \cA_\a u_n)_{L_2(\vw)} - \|\a\|_S \|u_n\|_{W_2^1(\vw)}^2=(f_n,u_n)_{W_2^1(\vw)},
\\
&\|\a\|_S=(\a u_n,v_n)_{L_2(S)}-(f_n,u_n)_{W_2^1(\vw)}.
\end{aligned}
\end{equation}

Проинтегрируем теперь по частям следующим образом:
\begin{equation}\label{12.50}
\begin{aligned}
\int\limits_{\vw} u_n\overline{v}\D \Psi^\e\,dx=&-\int\limits_{S} \frac{\p\Psi^\e}{\p\tau}  u_n\overline{v_n}\,ds+ \int\limits_{\p\om\setminus S} \frac{\p \Psi^\e}{\p\nu} u_n\overline{v_n}\,ds
 \\
 &-\int\limits_{\vw} \nabla \Psi^\e\cdot\nabla (u_n\overline{v}_n)\,dx.
\end{aligned}
\end{equation}
Справедливость этой формулы для произвольных $u_n, v_n\in W_2^1(\vw)$ можно строго проверить, выписав её вначале для бесконечно дифференцируемых функций $u_n$, $v_n$, а потом воспользовавшись плотностью этих множеств в пространстве $W_2^1(\vw)$.

Равенство (\ref{12.50}) перепишем теперь следующим образом:
\begin{align*}
(u_n\D \Psi^\e,v_n)_{L_2(\vw)}=&- \left( \frac{\p\Psi^\e}{\p\tau} u_n, v_n\right)_{L_2(S)}   + \left(\frac{\p \Psi^\e}{\p\nu} u_n, v_n\right)_{L_2(\p\vw\setminus S)}
\\
&- (u_n\nabla\Psi^\e,\nabla v_n)_{L_2(\vw)} -(\nabla u_n,v_n \nabla \Psi^\e)_{L_2(\vw)}
\\
=&-(\a u_n, v_n)_{L_2(S)}-\left(\left(\frac{\p\Psi^\e}{\p\tau}-\a\right) u_n, v_n\right)_{L_2(S)}
 \\
 &+ \left(\frac{\p \Psi^\e}{\p\nu} u_n, v_n\right)_{L_2(\p\vw\setminus S)}
- \big(\nabla (u_n\Psi^\e),\nabla v_n\big)_{L_2(\vw)}
\\
&- \big(\nabla u_n, \nabla (v_n \Psi^\e)\big)_{L_2(\vw)}
+ 2(\Psi^\e\nabla u_n,\nabla v_n)_{L_2(\vw)}.
\end{align*}
Отметим ещё, что в силу определения функции $w_n$ выполнено
\begin{equation*}
(\nabla u_n, \nabla v_n \Psi^\e)_{L_2(\vw)}=\|\a\|_{S}^{-1}(\nabla w_n,\nabla v_n \Psi^\e)_{L_2(\vw)}-\|\a\|_{S}^{-1}(\nabla f_n,\nabla v_n \Psi^\e)_{L_2(\vw)}.
\end{equation*}
Перепишем теперь полученное равенство, используя определение функции $w_n$, интегральное тождество для $u_n$ из (\ref{12.49}) с $\vp=v_n\Psi^\e$ и аналогичное тождество для $\vp=u_n\Phi^\e$:
\begin{equation}\label{12.51}
\begin{aligned}
(\a u_n, v_n)_{L_2(S)}=&-(u_n\D \Psi^\e,v_n)_{L_2(\vw)}  + \left(\frac{\p \Psi^\e}{\p\nu} u_n, v_n\right)_{L_2(\p\vw\setminus S)}
\\
&+ 2(\Psi^\e\nabla u_n,\nabla v_n)_{L_2(\vw)}
-\left(\left(\frac{\p\Psi^\e}{\p\tau}-\a\right) u_n, v_n\right)_{L_2(S)}
\\
&- (\Psi^\e u_n,\a u_n)_{L_2(S)}+ \|\a\|_{S}^{-1}(\nabla f_n,\nabla v_n \Psi^\e)_{L_2(\vw)}
\\
&-\|\a\|_{S}^{-1}(\a v_n, \Psi^\e v_n)_{L_2(S)}.
\end{aligned}
\end{equation}
Это равенство, неравенства (\ref{12.52}), (\ref{12.33}), определение величины $\mu(\e)$ из условия леммы и очевидная оценка
\begin{equation*}
\|u\|_{L_2(S)}+\|u\|_{L_2(\p\om\setminus S)}\leqslant C\|u\|_{W_2^1(\vw)},\qquad u\in W_2^1(\vw),
\end{equation*}
с константой $C$, не зависящей от $u$, позволяют оценить левую часть в (\ref{12.51}) следующим образом:
\begin{equation*}%\l%abel{12.82}
\big|(\a u_n, v_n)_{L_2(S)}\big|\leqslant C\mu(\e) \|\a\|_{L_\infty(S)} + C\|f_n\|_{W_2^1(\vw)}\|\Psi^\e\|_{W_\infty^1(\vw)},
\end{equation*}
где константа $C$ не зависит от $\e$, $n$, $u_n$, $v_n$, $f_n$, $\a$. Заменяя теперь скалярное произведение $(\a u_n, v_n)_{L_2(S)}$ на правую часть второго равенства в (\ref{12.84}) и переходя потом к пределу при $n\to+\infty$ с учётом сходимости в (\ref{12.48}), получаем:
\begin{equation*}
\|\a\|_S\leqslant C\mu(\e) \|\a\|_{L_\infty(S)}.
\end{equation*}
Применяя теперь
лемму~\ref{lm12.9},
приходим к (\ref{12.47}). Лемма доказана.
\end{proof}

Доказанная лемма даёт удобный способ проверки условия~\ref{A5}: достаточно отыскать функции $\a^0$ и $\Psi^\e$, удовлетворяющие условию (\ref{12.46}). Например, это легко сделать в случае строго периодической перфорации, когда
\begin{align}
&S=\{x:\, x_n=0\}, \qquad M_k^\e=\e(M_k+M), \nonumber
\\
&M_k:=(b_1 k_1,\ldots, b_{n-1} k_{n-1}), \qquad (k_1,\ldots,k_{n-1})\in\mathds{Z}^{n-1}=:\mathbb{M},\label{12.55}
\end{align}
где $b_i>0$ -- некоторые числа, $M$ -- некоторая точка в области $\Pi:=\square\times \mathds{R}$,
\begin{equation}\label{12.56} \square:=\Big\{x:-\frac{b_i}{2}<x_i<\frac{b_i}{2},\,i=1,\ldots,n-1\Big\}.
\end{equation}
Будем считать, что $\eta=1$, $\om_{k,\e}=\om$, где $\om\subset\mathds{R}^n$ -- некоторая фиксированная ограниченная область, такая что
$\overline{\om+M}\subset \Pi$. В этом случае функция $\a^\e$
имеет следующий вид:
\begin{equation*}
\a^\e(x)=
\frac{|\p\om|}{R_2^{n-1}}
\z\left(\frac{|x'-\e(M_k+M_\bot)|}{\e R_2}\right)
\end{equation*}
при $|x'-\e(M_k+M_\bot)|<\e R_2$, $k\in \mathds{Z}^{n-1}$, и $\a^\e=0$
в остальных точках поверхности $S$.
%\begin{equation*}
%\a^\e(x)=\left\{
%\begin{aligned}
%\frac{|\p\om|}{R_2^{n-1}}
%\z&\left(\frac{|x'-\e(M_k+M_\bot)|}{\e R_2}\right)
%&&\text{при}\quad
%\begin{aligned}
%&|x'-\e(M_k+M_\bot)|<\e R_2,
%\\
%&k\in \mathds{Z}^{n-1},
%\end{aligned}
%%
%%|x'-\e(M_k+M_\bot)|<\e R_2,\quad
%%\\
%%&&&\hphantom{\text{при}}\quad k\in \mathds{Z}^{n-1},
%\\
%& 0 \quad &&\text{в остальных точках поверхности}\ S,
%\end{aligned}
%\right.
%\end{equation*}
Здесь $M_\bot$ -- проекция точки $M$ на плоскость  $x_n=0$ и $x'=(x_1,\ldots,x_{n-1})$, а константа $R_2$ выбрана из условия $\overline{B_{R_2}(M)}\subset\Pi$. В качестве $\a^0$ возьмём постоянную функцию:
\begin{equation*}
\a^0(x'):= \frac{|\p\om|}{|\square|},\qquad x'\in S.
\end{equation*}
Тогда существует бесконечно дифференцируемое $\square$-периодическое решение краевой задачи
\begin{equation*}
\D_\xi \Psi=0\quad\text{при}\quad \xi_n>0, \qquad \frac{\p\Psi}{\p \xi_n}=\a^0-\a^\e(\e\xi')\quad\text{при}\quad \xi_n=0,
\end{equation*}
равномерно экспоненциально убывающее при $\xi_n\to+\infty$, где $\xi:=(\xi_1,\ldots,\xi_n)$. Теперь достаточно  положить $\Psi^\e(x):=\e\Psi(\frac{x}{\e})$ и сразу видим, что условие (\ref{12.46}) выполнено с $\mu(\e)=C\e$, где $C$ -- некоторая константа, не зависящая от $\e$.

Такой же подход удаётся перенести и на более общий случай локально периодических перфораций вдоль гладких поверхностей. А именно, пусть  поверхность $S$ имеет гладкость $C^5$ и удовлетворяет условию~\ref{A1}. На этой поверхности зададим разбиение единицы $1=\sum\limits_{p\in\mathds{N}} \z_p$ и пусть $\supp\z_p\Subset S_p$, $p\in\mathds{N}$, где $S_p$ -- некоторые открытые односвязные компактные части поверхности $S$ с гладкими краями. Будем считать, что каждая точка поверхности $S$ попадает в конечное число множеств $S_p$ и это число ограничено равномерно по всем точкам поверхности. Предположим ещё, что для каждой части $S_p$ поверхности $S$ существует  дифференцируемый диффеоморфизм $\cP_p$ класса гладкости $C^5$, отображающий некоторую фиксированную односвязную ограниченную область $\overline{D}\Subset \mathds{R}^{n-1}$, содержащую нуль, на кусок поверхности $\overline{S_p}$, причём якобианы обоих отображений $\cP_p$ и $\cP_p^{-1}$ ограничены равномерно сверху и снизу как по пространственным переменным, так и по параметру $p$.  Через $s$ обозначим декартовы переменные на множестве $D$ и их будем использовать в качестве локальных переменных на каждой из частей $S_p$.

Следующие условия описывают локально-периодическую структуру перфорации. А именно, предположим, что
\begin{equation}\label{12.57}
\big\{M_{k,\bot}^\e\,:\, M_{k,\bot}^\e\in S_p\big\}=\big\{\cP_p(\e M_k)\,:\,k\in\mathds{Z}^{n-1}\big\}\cap \mathds{R}^{n-1},
\end{equation}
где $M_k$ -- точки из (\ref{12.55}).
Относительно площадей границ полостей будем считать, что
\begin{equation}\label{12.58}
|\p\om_{k,\e}|=\bw_p(\e M_k,\e),
\end{equation}
где индекс $k$ выбирается из условия $M_{k,\bot}^\e\in S_p$, а $\bw_p=\bw_p(s,\e)$, $s\in D$ -- некоторая  функция,
такая, что
\begin{equation*}%\%label{12.79}
\bw_p(\,\cdot\,0)\in C^5(\overline{D}),\qquad
 \sup\limits_{p\in\mathds{N}} \|\bw_p(\,\cdot\,,\e)-\bw_p(\,\cdot\,,0)\|_{L_\infty(S_p)}\to0,\quad \e\to+0.
\end{equation*}

Оказывается, что указанных условий на поверхность и структуру перфораций достаточно, чтобы гарантировать выполнение условия~\ref{A5}. При этом свойство локальной периодичности выражается условием (\ref{12.57}) о {ло\-каль\-но}-{пе\-рио\-ди\-чес\-ком} распределении точек $M_k^\e$ с точностью до диффеоморфизмов $\cP_p$ и существованием гладкой функции $\bw$, описывающей площади полостей в смысле равенства (\ref{12.58}). Выполнение условия (\ref{A5}) далее будет строго доказано в лемме~\ref{lm12.8}. Для строгой формулировки этой леммы нам понадобится ещё один вспомогательный объект, связанный с диффеоморфизмами $\cP_p$.

Так как диффеоморфизм $\cP_p$ гладкий, то в каждой точке $y\in D$ справедливо равенство
\begin{equation*}%\l%abel{12.60}
\cP_p(y+s)=\cP_p(y)+\cP_p'(y)s + \tilde{\cP}_p(y,s),
\end{equation*}
где $\cP'(y)$ -- некоторый линейный оператор на пространстве $\mathds{R}^{n-1}$, зависящий от $y$ и имеющий гладкость класса $C^3$ по этой переменной, а $\cP_p(y,s)$ -- гладкое отображение класса $C^2$, удовлетворяющее равномерному по $p$, $s$, $y$ равенству
\begin{equation*}%\l%abel{12.61}
\tilde{\cP}_p(y,s)=O(|s|^2),
\end{equation*}
которое допускает дифференцирование по $s$ и $y$.

\begin{lemma}\label{lm12.8}
При выполнении описанных выше условий на поверхность $S$ и условий локально-периодичности перфорации существует функция $\Psi^\e$, для которой выполнены условия леммы~\ref{lm12.4} с
\begin{equation}\label{12.62}
\begin{aligned}
&\a^0(x):=\frac{\eta^{n-1}\bw(\cP_p^{-1}x,0)}{|\square|} \int\limits_{\mathds{R}^{n-1}} \z\big(|\cP_p'(\cP_p^{-1}x)\xi|\big)\,d\xi,
\\
&\mu(\e):=C\left(\sup\limits_{p\in\mathds{N}} \|\bw_p(\,\cdot\,,\e)-\bw_p(\,\cdot\,,0)\|_{L_\infty(S_p)} +\e\eta^{n-1}\right),
\end{aligned}
\end{equation}
где константа $C$ не зависит от $\e$.
\end{lemma}

\begin{proof}
С учётом сделанных предположений, для каждой точки $M_{k,\bot}^\e$
при $|x-M_{k,\bot}^\e|<\e R_2$ имеем:
\begin{equation*}
|x-M_{k,\bot}^\e|=|\cP_p'(y)(y-\e M_k)|+O(\e^2),\qquad x=\cP_p y,
\end{equation*}
и потому для таких $x$ выполнено
\begin{align*}
&
\z\left(\frac{|x-M_{k,\bot}^\e|}{\e R_2}\right)= \z\left(\frac{|\cP_p'(y)(y-\e M_k)|}{\e R_2}\right)+O(\e),
\\
&\bw_p(\e M_k,\e)=\bw_p(y,0)+O\Big(\e+ \sup\limits_{p\in\mathds{N}} \|\bw_p(\,\cdot\,,\e)-\bw_p(\,\cdot\,,0)\|_{L_\infty(S_p)}
\Big).
\end{align*}
Поэтому, используя разбиение единицы функциями $\z_p$, функцию $\a^\e$ можно представить в виде
\begin{equation*}%\l%abel{12.63}
\a^\e(x)=\eta^{n-1}\a_0^\e(x)+\e\tilde{\a}^\e(x),
\end{equation*}
где функция $\tilde{\a}^\e$ удовлетворяет оценке
\begin{equation*}%\l%abel{12.64}
\|\tilde{\a}^\e\|_{L_\infty(S)}\leqslant C\eta^{n-1},
\end{equation*}
$C$ -- константа, не зависящая от $\e$, а функция $\a_0^\e$ имеет вид
\begin{equation*}%\l%abel{12.65}
\a_0^\e(x)=\sum\limits_{p\in\mathds{N}} \a_p^\e(y),
\end{equation*}
где
\begin{equation*}
\a_p^\e(y):= \chi_p(\cP_p(y))
\frac{\bw_p(y,0)}{R_2^{n-1}}
\z\left(\frac{|\cP_p'(y)(y-\e M_k)|}{\e R_2}\right)
\end{equation*}
при $|\cP_p'(y)(y-\e M_k)|<\e R_2$, и $\a_p^\e(y):=0$
 в остальных точках $D$.
Определим ещё области
\begin{equation*}
\vw_p:=\Big\{x\in\vw\,:\, x=\cP_p y+\tau\nu(\cP_p(y)),\, y\in D, \, \tau\in(0,\tfrac{\tau_0}{2})\Big\}.
\end{equation*}
Ясно, что $\vw=\bigcup\limits_{p\in\mathds{N}}\vw_p$. Рассмотрим теперь краевые задачи
\begin{equation}\label{12.66}
\begin{gathered}
\D_x \Psi_p^\e=0\quad\text{в}\quad \vw_p,
\\
\frac{\p\Psi_p^\e}{\p\tau}=\z_p \a^0-\a_0^\e\quad\text{на}\quad S_p,
\qquad \frac{\p \Psi_p^\e}{\p\nu}=0\quad\text{на}\quad \p\vw\cap\p\vw_p.
\end{gathered}
\end{equation}
Основная идея доказательства состоит в построении формального асимптотического решения такого семейства задач с последующей их склейкой с помощью разбиения единицы:
\begin{equation}\label{12.67}
\Psi^\e=\eta^{n-1}\sum\limits_{p\in\mathds{N}} \z_p\Psi_p^\e.
\end{equation}

Формальное асимптотическое решение задачи (\ref{12.66}) будем строить методом двух масштабов в следующем виде:
\begin{equation}\label{12.68}
\Psi_p^\e(x)=\e\Psi_p^{(0)}(y,\xi) + \e^2 \Psi_p^{(1)}(y,\xi), \qquad
\xi=(\xi',\xi_n):=(y\e^{-1},\tau\e^{-1}).
\end{equation}
Функции $\Psi_p^{(0)}$ будем искать $\square$-периодическими, где множество $\square$ было определено в (\ref{12.56}).

Оператор Лапласа в переменных $(y,\tau)$ переписывается следующим образом:
\begin{equation*}%\l%abel{12.69}
\D_x=\frac{\p^2\ }{\p\tau^2} + \ell_n(y,\tau)\frac{\p\ }{\p\tau} + \sum\limits_{i,j=1}^{n-1} \ell_{ij}(y,\tau) \frac{\p^2\ }{\p y_i\p y_j} + \sum\limits_{i=1}^{n-1}\ell_i(y,\tau)\frac{\p\ }{\p y_j},
\end{equation*}
где $\ell_n$, $\ell_i$, $\ell_{ij}$ -- некоторые функции, причём $\ell_n, \ell_i\in C^2(\overline{D}\times[0,\frac{\tau_0}{2}])$, $\ell_{ij}\in C^3(\overline{D}\times[0,\frac{\tau_0}{2}])$, и для функций $\ell_{ij}$ выполнено условие равномерной эллиптичности:
\begin{equation*}%\l%abel{12.70}
\sum\limits_{i,j=1}^{n-1} \ell_{ij}(y,\tau) z_i z_j\geqslant c_3 \sum\limits_{i=1}^{n-1} z_i^2,\qquad (z_1,\ldots,z_{n-1})\in\mathds{R}^{n-1}
\end{equation*}
с константой $c_3>0$, не зависящей от $y$, $\tau$, $z_1$, \ldots, $z_{n-1}$. Тогда с учётом определения переменных $\xi$,
оператор $\D_x$ на функциях $\Psi=\Psi(y,\xi)$  переписывается следующим образом:
\begin{equation}\label{12.71}
\begin{aligned}
&\D_x \Psi(y,\xi)=\e^{-2} \cL_{-2} + \e^{-1} \cL_{-1} + \cL_\e,
\\
&\cL_{-2}:=\frac{\p^2\ }{\p\xi_n^2} + \sum\limits_{i,j=1}^{n-1}
\ell_{ij}(y,0) \frac{\p^2\ \ }{\p\xi_i \p \xi_j},
\\
&\cL_{-1}:=\ell_n(y,0) \frac{\p\ }{\p\xi_n} + 2\sum\limits_{i,j=1}^{n-1}\ell_{ij}(y,0) \left(\frac{\p^2\ \ }{\p\xi_i \p y_j}+\frac{\p^2\ \ }{\p y_i \p\xi_j}\right)
\\
&\hphantom{\cL_{-1}:=}+
\sum\limits_{i,j=1}^{n-1}\frac{\p\ell_{ij}}{\p\tau}(y,0)\xi_n \frac{\p^2\ \ }{\p\xi_i \p \xi_j} +\sum\limits_{i=1}^{n-1} \ell_i(y,0) \frac{\p\ }{\p\xi_i},
\end{aligned}
\end{equation}
где $\cL_\e$ -- некоторый дифференциальный оператор по переменным $(y,\xi)$ с коэффициентами, ограниченными величиной $C|\xi_n|^2$ равномерно по $y$, $\xi$, $\e$. Теперь подставим полученное выражение для оператора Лапласа и (\ref{12.68}) в краевую задачу (\ref{12.66}) и соберём члены при двух старших степенях $\e$. Тогда для $\Psi_0$ и $\Psi_1$ получаем следующие краевые задачи:
\begin{align}\label{12.72}
&\cL_{-2}\Psi_p^{(0)}=0\quad\text{при}\quad \xi_n>0,\qquad
\frac{\p\Psi_p^{(0)}}{\p\xi_n}=\beta_p(y,\xi)\qquad\text{при}\quad\xi_n=0,
\\
&\cL_{-2}\Psi_p^{(1)}=-\cL_{-1}\Psi_p^{(0)}\quad\text{при}\quad \xi_n>0,\qquad
\frac{\p\Psi_p^{(1)}}{\p\xi_n}=0\qquad\text{при}\quad\xi_n=0, \label{12.73}
\end{align}
где обозначено
\begin{gather*}
\beta_p(y,\xi')=\chi_p(\cP_p y) \big(\beta_p^0(y,\xi')-\a^0(\cP_p y )\big),
\\
\beta_p^0(y,\xi'):=% \chi_p(\cP_p(y))
\left\{
\begin{aligned}
\frac{\bw_p(y,0)}{R_2^{n-1}}
\z&\left(\frac{|\cP_p'(y)\xi'|}{R_2}\right)
&&\text{при}\quad |\cP_p'(y)\xi'|< R_2,
\\
& 0 \quad &&\text{в остальных точках}\ \xi'.
\end{aligned}
\right.
\end{gather*}
При выводе задач (\ref{12.72}), (\ref{12.73}) краевое условие в (\ref{12.66}) при $\tau=\frac{\tau_0}{2}$ заменяется на условие экспоненциального убывания функций $\Psi_p^{(0)}$, $\Psi_p^{(1)}$ при $\xi_n\to+\infty$. Отметим ещё, что в силу определения функции $\a^0$ выполнено равенство
\begin{equation*}%\l%abel{12.75}
\int\limits_{\square} \beta_p(y,\xi)\,d\xi'=0,\qquad y\in\overline{D},\qquad \xi':=(\xi_1,\ldots,\xi_{n-1}).
\end{equation*}
Это условие обеспечивает разрешимость задачи (\ref{12.72}) в классе $\square$-{пе\-рио\-ди\-чес\-ких} по $\xi$ функций, экспоненциально убывающих при $\xi_n\to+\infty$ и удовлетворяющих условию
\begin{equation}\label{12.76}
\int\limits_{\square} \Psi_p^{(0)}(y,\xi)\,d\xi'=0,\qquad \xi_n>0,\qquad y\in\overline{D}.
\end{equation}
Решение этой задачи можно построить явно методом разделения переменных:
\begin{gather}\label{12.77}
\Psi_p^{(0)}(y,\xi)=\sum\limits_{\bn\in (b_1\mathds{Z}\times \cdots \times b_{n-1} \mathds{Z})\setminus\{0\}} \g_\bn^{(0)}(y) e^{-2\pi (\L_\bn(y)\xi_n -\iu\bn\cdot\xi')},
\\
\L_\bn(y):=\left(\sum\limits_{i,j=1}^{n-1} \ell_{ij}(y,0)\bn_i\bn_j\right)^{\frac{1}{2}},  \qquad  \g_\bn^{(0)}(y):=\frac{1}{|\square|}\int\limits_{\square}
\b_p(y,\xi')e^{-2\pi\iu \bn\cdot\xi'}\,d\xi',\nonumber
\end{gather}
где $\bn=(\bn_1,\ldots,\bn_{n-1})$. Коэффициенты $\g_p^{(0)}(y)$ перепишем в терминах преобразования Фурье функции $\z(|\,\cdot\,|)$ следующим образом:
\begin{align*}
\g_\bn^{(0)}(y):=&\frac{\chi_p(\cP_p(y))\bw_p(y,0)}{|\square|R_2^{n-1}}
\int\limits_{\square}
\z\left(\frac{|\cP_p'(y)\xi'|}{R_2}\right)e^{-2\pi\iu \bn\cdot\xi'}\,d\xi'
\\
=&\frac{\chi_p(\cP_p(y))\bw_p(y,0)}{|\square|\det \cP_p'(y)}
\int\limits_{\mathds{R}^{n-1}}
\z(|\xi'|)e^{-2\pi\iu R_2 \bn\cdot(\cP_p'(y))^{-1}\xi'}\,d\xi'
\\
=&\frac{\chi_p(\cP_p(y))\bw_p(y,0)}{|\square|\det \cP_p'(y)}\hat{\z}\big(2\pi R_2 \bn\cdot(\cP_p'(y))^{-1}\big),
\\
\hat{\z}(t):=&\int\limits_{\mathds{R}^{n-1}} \z(|\xi'|)e^{-\iu\xi'\cdot t}\,d\xi'.
\end{align*}
Так как функция $\z(|\,\cdot\,|)$ бесконечно дифференцируемая и финитная, её преобразование Фурье убывает на бесконечности быстрее любой обратной степени модуля аргумента, а потому, с учётом предположений относительно диффеоморфизмов $\cP_p$, коэффициенты $\g_\bn^{(0)}(y)$ принадлежат $C^3(\overline{D})$ и  убывают быстрее любой обратной степени индекса $|\bn|$ при его стремлении к бесконечности равномерно по $y$ и тоже самое верно для всех имеющихся производных этих коэффициентов по $y$. Этот факт обеспечивает сходимость ряда в (\ref{12.77}) в $C^2(\overline{\square\times\mathds{R}})$-норме. Ясно, что построенная таким образом функция $\Psi_p^{(0)}$  бесконечно дифференцируемая по  $\xi$ при $\xi_n\geqslant 0$  и принадлежит  $C^3(\overline{D})$ как функция переменной $y$ для каждого $\xi$. Кроме того, эта функция $\square$-периодична по $\xi$
 и вместе со всеми своими производными по  $y$ вплоть до третьего порядка принадлежит классу $C(\{\xi:\, \xi_n\geqslant 0\}\times\overline{D})$. Функция $\Psi_p^{(0)}(y,\xi)$ и все её производные по  $\xi$ и $y$ экспоненциально убывают при $\xi_n\to+\infty$ равномерно по  $\xi$ и $y$.

Свойство $\square$-периодичности функции $\Psi_p^{(0)}$ и условие (\ref{12.76}) обеспечивают эти же свойство и условие для правой части уравнения в (\ref{12.73}). А именно, уравнение в (\ref{12.73}) можно переписать в виде:
\begin{equation*}
\cL_{-2}\Psi_p^{(1)}=\sum\limits_{\bn\in (b_1\mathds{Z}\times \cdots \times b_{n-1} \mathds{Z})\setminus\{0\}} \rf_\bn(y,\xi_n)
e^{-2\pi (\L_\bn(y)\xi_n -\iu\bn\cdot\xi')}
\quad\text{при}\quad \xi_n>0,
\end{equation*}
где $\rf_\bn$ -- некоторые полиномы по $\xi_n$ степени не выше двух с коэффициентами, зависящими от $y$ и принадлежащими классу $C^2(\overline{D})$. Коэффициенты полиномов $\rf_\bn(y,\xi_n)$  убывают быстрее любой обратной степени индекса $|\bn|$ при его стремлении к бесконечности равномерно по $y$  и тоже самое верно для всех имеющихся производных этих коэффициентов по $y$. Ещё отметим, что эти коэффициенты обращаются в нуль в точках $y$ вне носителя функции $\z_p(\cP_p y)$.

Такой вид правой части обеспечивает разрешимость задачи (\ref{12.73}) в нужном классе функций и позволяет найти решение в явном виде:
\begin{equation*}
\Psi_p^{(1)}(y,\xi)=\sum\limits_{\bn\in (b_1\mathds{Z}\times \cdots \times b_{n-1} \mathds{Z})\setminus\{0\}} \g_\bn^{(1)}(y,\xi_n) e^{-2\pi (\L_\bn(y)\xi_n -\iu\bn\cdot\xi')},
\end{equation*}
где $\g_\bn^{(1)}=\g_\bn^{(1)}(y,\xi_n)$ -- некоторые полиномы по $\xi_n$ степени не выше третьей с коэффциентами, зависящими  от $y$ и принадлежащими классу $C^2(\overline{D})$, причём $\g_\bn^{(1)}(y,0)=0$ для всех $\bn$. Коэффициенты полиномов $\g_\bn^{(1)}(y,\xi_n)$  убывают быстрее любой обратной степени индекса $|\bn|$ при его стремлении к бесконечности равномерно по $y$  и тоже самое верно для всех имеющихся производных этих коэффициентов по $y$. Эти коэффициенты обращаются в нуль в точках $y$ вне носителя функции $\z_p(\cP_p y)$.

Ясно, что построенная  функция $\Psi_p^{(1)}$  бесконечно дифференцируемая по  $\xi$ при $\xi_n\geqslant 0$  и принадлежит  $C^2(\overline{D})$ как функция переменной $y$ для каждого $\xi$. Эта функция $\square$-периодична по $\xi$ и вместе со всеми своими производными по  $y$ вплоть до второго порядка принадлежит классу $C(\{\xi:\, \xi_n\geqslant 0\}\times\overline{D})$. Функция $\Psi_p^{(1)}(y,\xi)$ и все её производные по  $\xi$ и $y$ экспоненциально убывают при $\xi_n\to+\infty$ равномерно по  $\xi$ и $y$.

Проверим теперь, что функция, определённая формулой (\ref{12.68}), удовлетворяет предположениям леммы~\ref{lm12.4}. В силу краевых задач (\ref{12.72}), (\ref{12.73}) и соотношений (\ref{12.71}) сразу видим, что эта функция удовлетворяет краевому условию на $S$ из задачи (\ref{12.66}) и равенству
\begin{equation*}
\D_x \Psi_p^\e =\e \big(\cL_{-1}\Psi_1^{(p)}+\cL_\e(\Psi_0^{(p)}+\e \Psi_1^{(p)})\big) \quad \text{в} \quad \vw_p.
\end{equation*}
Поэтому
\begin{equation*}%\l%abel{12.81}
\|\D_x \Psi_p^\e\|_{L_\infty(\vw_p)}\leqslant C\e.
\end{equation*}
где константа $C$ не зависит от $\e$. Из равномерного экспоненциального убывания функций $\Psi_p^{(0)}$ и $\Psi_p^{(1)}$ элементарно выводим, что
\begin{equation*}
\left\|\frac{\p\Psi_p^\e}{\p\nu}\right\|_{L_\infty(\p\vw\cap\p\vw_p)}
\leqslant Ce^{-\frac{c}{\e^2}},
\end{equation*}
где $c$, $C$ -- некоторые положительные константы, не зависящие от $\e$ и $p$. Отметим ещё, что функции $\Psi_p^\e$ тождественно обращаются в нуль вне носителя функции $\z_p(\cP_p(y))$.  Используя теперь установленные факты  о функциях $\Psi_p^\e$, легко видим, что функция
$\Psi^\e$, определённая формулами (\ref{12.67}), (\ref{12.68}), удовлетворяет условиям леммы~\ref{lm12.4} с $\a^0$ и $\mu$ из (\ref{12.62}). Лемма доказана.
\end{proof}

\section{Вспомогательные леммы}\label{s3}
В данном параграфе мы докажем ряд вспомогательных лемм, которые далее будут использоваться в доказательстве наших основных теорем.

\begin{lemma}\label{lm3.1}
При выполнении условий~\ref{A1},~\ref{A2},~\ref{A3} для любой функции $u\in W^1_2(\Om^\e)$ верна оценка
\begin{equation*}%\l%abel{3.1}
\|u\|_{L_2(\p\tht^\e)}^2\leqslant C(\e\eta+\d\eta^{n-1})\|\nabla u\|_{L_2(\Om^\e)}^2 +C(\d)\eta^{n-1}\|u\|^2_{L_2(\Om^\e)},
\end{equation*}
где $\d>0$ -- произвольная константа, а константы $C$ и $C(\d)$ не зависят от параметров $\e$, $\eta$, функции $u$, а также от формы и расположения полостей $\om_k^\e$, $k\in\mathbb{M}^\e$.
\end{lemma}

Доказательство этой леммы почти дословно совпадает с доказательством леммы 3.4 из \cite{25}. В доказательстве леммы 3.4 из \cite{25} функция $u$ была продолжена нулем внутрь полостей c первым граничным условием, и далее при доказательстве леммы наличие полостей с первым граничным условием не использовалось. В нашем случае таких полостей нет, поэтому никаких продолжений делать не требуется.

\begin{lemma}\label{lm3.2}
Пусть выполнены условия~\ref{A1},~\ref{A2},~\ref{A3},~\ref{A4}. Тогда существует $\lambda_0$, не зависящее от $\e$,  такое что при $\lambda<\lambda_0$ для всех $f\in L_2(\Om)$ задачи (\ref{2.5}), (\ref{2.12}) и (\ref{2.13}), (\ref{2.15}) имеют единственное решения $u_0\in W_2^1(\Omega)$ и $u_\e\in W_2^1(\Omega^\varepsilon)$ для всех достаточно малых $\e$.

При $\l<\l_0$ для всех $u\in W_2^1(\Om^\e)$ верна априорная оценка
\begin{equation}\label{4.9}
\big|\mathfrak{h}_0(u,u)-\l\|u\|_{L_2(\Om^\e)}^2\big|\geqslant C\|u\|_{W_2^1(\Om^\e)}^2,
\end{equation}
где константа $C$  не зависит от $u$ и $\e$.

Для решения задачи (\ref{2.5}) верна  равномерная оценка
\begin{equation}\label{3.23}
\|u_\e\|_{W_2^1(\Om^\e)}\leqslant C\|f\|_{L_2(\Om^\e)}.
\end{equation}
где константа $C$ не зависит от $f$ и $\e$.

Решение задачи (\ref{2.12}) является элементом пространства $W_2^2(\Om)$ и верна равномерная оценка
\begin{equation}\label{4.7}
\|u_0\|_{W_2^2(\Om)}\leqslant C\|f\|_{L_2(\Om)},
\end{equation}
где константа $C$ не зависит от $f$.

Решение задачи (\ref{2.13}), (\ref{2.15}) является элементом пространства $W_2^2(\Om\setminus S)$ и верна равномерная оценка
\begin{equation}\label{4.7a}
\|u_0\|_{W_2^1(\Om)}+\|u_0\|_{W_2^2(\Om\setminus S)}\leqslant C\|f\|_{L_2(\Om)},
\end{equation}
где константа $C$ не зависит от $f$.
\end{lemma}

\begin{proof}
Доказательство этой леммы в целом проводится по схеме доказательства леммы 5.1 из \cite{25} и отличается от последнего лишь в некоторых деталях. Поэтому кратко опишем схему доказательства и остановимся на имеющихся отличиях.

Мы обсудим только задачу (\ref{2.5}), так как  для задачи (\ref{2.13}), (\ref{2.15}) доказательство проводится совершенно аналогично, а задача
(\ref{2.12}) является частным случаем задачи (\ref{2.13}), (\ref{2.15}), соответствующему равенству $\a^0=0$.

Вначале в пространстве $\Ho^1(\Om^\e, \partial\Om)$ необходимо ввести оператор, действующий по правилу:
каждой функции $u\in\Ho^1(\Om^\e,\partial\Om)$ ставится в соответствие линейный непрерывный
функционал, заданный на $W_2^{1}(\Om^\e)$ и действующий по правилу $v\mapsto\mathfrak{h}_a(u,v)$, $v\in\Ho^1(\Om^\e,\p\Om)$.
Далее для доказательства однозначной разрешимости задачи (\ref{2.5}) достаточно проверить выполнение следующих свойств \cite[Гл. V\!I, \S 18.4]{Vain}, \cite[Гл. 1, \S 1.2$^0$]{Dub}:
\begin{enumerate}
\item\label{1} Для любых $u,v,w\in\Ho^1(\Om^\e,\p\Om)$ функция $t\mapsto\mathfrak{h}_a(u+t  v,w)$ непрерывна;
\item\label{2} Для любых $u,v\in\Ho^1(\Om^\e,\p\Om)$ выполнено $\RE\big(\mathfrak{h}_a(u,u-v)-\mathfrak{h}_a(v,u-v)\big)>0$;
\item\label{3} Справедливо соотношение
\begin{equation*}%\l%abel{3.12}
\frac{\RE\mathfrak{h}_a(u,u)}{\|u\|_{W_2^1(\Om^\e)}}\to+\infty,\qquad \|u\|_{W_2^1(\Om^\e)}\to+\infty.
\end{equation*}
\end{enumerate}

Свойство \ref{1} проверяется аналогично проверке соответствующего свойства из доказательства леммы 5.1 в \cite{25}.

Проверим свойство \ref{2}. Сразу же отметим равенство
\begin{equation}\label{3.16}
\mathfrak{h}_a(u,u-v)-\mathfrak{h}_a(v,u-v))=\mathfrak{h}_0(u-v,u-v)+
\big(a(\,\cdot\,,u)-a(\,\cdot\,,v),u-v\big)_{L_2(\p\tht^\e)}
\end{equation}
и следующую тривиальную оценку:
\begin{equation}\label{3.11}
\bigg|\sum\limits_{j=1}^n\left( A_j\frac{\p
u}{\p x_j},u\right)_{L_2(\Om^\e)}+(A_0 u,u)_{L_2(\Om^\e)}
\bigg|\leqslant \frac{c_0}{4} \|\nabla u\|_{L_2(\Om^\e)}^2 + C_1 \|u\|_{L_2(\Om^\e)}^2,
\end{equation}
где $C_1$ -- некоторая константа, не зависящая от $u\in W_2^1(\Om^\e)$ и $\e$, а константа $c_0$ введена в условии эллиптичности (\ref{2.1}). Отсюда и из условия эллиптичности уже вытекает оценка (\ref{4.9}), если взять $\l<-C_1-1$.

Так как функция $a$ имеет ограниченные производные по $\RE u$ и  $\IM u$ (см. второе условие в (\ref{2.2})), то она удовлетворяет оценке:
\begin{equation}\label{3.2}
|a(x,u)-a(x,v)|\leqslant a_0 |u-v|,
\end{equation}
где $a_0$ -- некоторая константа, не зависящая от $x$  и $u$.
Поэтому в силу леммы~\ref{lm3.1} верно неравенство
\begin{equation*}%\l%abel{3.3}
\big|(a(\,\cdot\,,u-v),u-v)_{L_2(\p\tht^\e)}\big|\leqslant \frac{c_0}{4} \|\nabla (u-v)\|_{L_2(\Om^\e)}^2 + C_2 \|u-v\|_{L_2(\Om^\e)}^2,
\end{equation*}
где константа $C_2$ не зависит от $\e$, $u,v\in W_2^1(\Om^\e)$. Учитывая эту оценку и (\ref{3.16}), (\ref{3.11}) и полагая $\l<\l_0:=-C_1-C_2-\frac{c_0}{4}$,
получаем:
\begin{equation}\label{3.20}
\begin{aligned}
\RE \big(\mathfrak{h}_a(u,u-v)-\mathfrak{h}_a(v,u-v))\big)\geqslant &
\frac{c_0}{2}\|\nabla (u-v)\|_{L_2(\Om^\e)}^2
\\
&-(\l+C_1+C_2)\|u-v\|_{L_2(\Om^\e)}^2
\\
\geqslant & \frac{c_0}{4} \|u-v\|_{W_2^1(\Om^\e)}^2.
\end{aligned}
\end{equation}
Из этой оценки уже следует свойство~\ref{2}. Полагая в этой оценке $v=0$ и учитывая равенство $\mathfrak{h}_a(0,u)=0$, сразу приходим к свойству~\ref{3}.

Аналогично выводу оценки (\ref{3.20}) несложно проверить, что
\begin{equation*}
\frac{c_0}{4} \|u_\e\|_{W_2^1(\Om^\e)}^2\leqslant |\mathfrak{h}_a(u^\e,u^\e)|=\big|(f,u_\e)_{L_2(\Om^\e)}\big|\leqslant
\|f\|_{L_2(\Om^\e)}\|u^\e\|_{L_2(\Om^\e)},
\end{equation*}
откуда вытекает априорная оценка (\ref{3.23}) с $C=\frac{4}{c_0}$.

Однозначная разрешимость задач (\ref{2.12}) и (\ref{2.13}), (\ref{2.15}) устанавливаются аналогично. Для решений этих задач верны априорные оценки, аналогичные (\ref{3.23}) с заменой пространств $W_2^1(\Om^\e)$ и $L_2(\Om^\e)$ на $W_2^1(\Om)$ и $L_2(\Om)$.

Уравнение в задаче (\ref{2.12}) можно переписать в виде
\begin{equation}\label{3.26}
 -\sum\limits_{i,j=1}^n\frac{\p}{\p x_i}A_{ij}\frac{\p u_0}{\p x_j}=f-\sum\limits_{j=1}^n
A_j\frac{\p}{\p x_j}
-(A_0+\lambda)u_0=f\quad\text{в}\quad\Om,
\end{equation}
где в силу априорных оценок для решения, аналогичных (\ref{3.23}),  правая часть -- элемент пространства $L_2(\Om)$, чья норма оценивается через $C\|f\|_{L_2(\Om)}$ с константой $C$, не зависящей от $f$.
Учитывая теперь краевое условие из задачи (\ref{2.12}), в силу стандартных теорем о повышении гладкости приходим к неравенству (\ref{4.7}).

Оценка (\ref{4.7a}) доказывается аналогично с единственным отличием, что здесь помимо приведения уравнения к виду (\ref{3.26}), необходимо ещё
правую часть в краевом условии на скачок производной считать следом на $S$ заданной функции из $W_2^1(\Om)$.
Лемма доказана.
\end{proof}

Обозначим: $B_r^k:=B_{rR_2\e\eta}(M_k^\e)$.

\begin{lemma}\label{lm3.3}
При выполнении условий \ref{A1},~\ref{A2},~\ref{A3} для любой функций $u\in W_2^1(\Om^\e)$  выполнено неравенство:
\begin{equation*}%\l%abel{3.3}
\sum\limits_{k\in\mathbb{M}^\e}\|u\|_{L_2(B_{b_*}^k\setminus\om_k^\e)}^2 \leqslant C(\e^2\eta^2\|\nabla u\|_{L_2(\Om^\e)}^2+\e\eta^n\|u\|_{L_2(\Om^\e)}^2),
\end{equation*}
где С -- некоторая константа, не зависящая от параметров $k$, $\e$, $\eta$, функции $u$, формы и расположения полостей $\om_k^\e$, $k\in\mathbb{M}^\e$.
\end{lemma}

\begin{proof}
Всюду в доказательстве через $C$ обозначаем различные несущественные константы, не зависящие от $k$, $\e$, $\eta$, $u$, формы и расположения полостей $\om_k^\e$, $k\in\mathbb{M}^\e$. Напомним, что $\chi=\chi(t)$ -- бесконечно дифференцируемая срезающая функция, равная единице при $t<1$ и нулю при $t>2$. В окрестности каждой из точек $M_k^\e$ введем растянутые координаты по правилу: $y=(x-M_k^\e)\e^{-1}$.
Обозначим:
\begin{equation*}
\tilde{u}(y):=u(M_k^\e+\e\eta y)\chi\left(\frac{2|y|}{(b+1)R_2\eta}\right),
\end{equation*}
$\tilde{\om}_{k,\e}$ -- область, полученная сжатием $\omega_{k,\e}$ в $\eta^{-1}(\e)$ раз. Функция $\tilde{u}$ является элементом пространства $\Ho^1(B_{(b+1)R_2\eta}(0)\setminus\tilde{\om}_{k,\e},\p B_{(b+1)R_2\eta}(0))$.
В силу леммы 3.1 из \cite{25} выполнено неравенство:
\begin{align*}
\|u(M_k^\e+\e\eta\,\cdot\,)\|_{L_2(B_{b_*R_2\eta}(0) \setminus\tilde{\om}_{k,\e})}^2 &\leqslant \|\tilde{u}\|_{L_2(B_{(b+1)R_2\eta}(0)\setminus\tilde{\om}_{k,\e})}^2
\\
\leqslant &  C\eta^2\|\nabla_y\tilde{u}\|_{L_2(B_{(b+1)R_2\eta}(0) \setminus\tilde{\om}_{k,\e})}^2
\\
\leqslant& C \eta^2\Big(\|\nabla_y u(M_k^\e+\e\eta\,\cdot)\|_{L_2(B_{(b+1)R_2\eta}(0) \setminus\tilde{\om}_{k,\e})}^2
\\
&
+\eta^{-2}\|u(M_k^\e+\e\eta\,\cdot\,)\|_{L_2(B_{(b+1)R_2\eta}(0)\setminus B_{b_{*}R_2\eta}(0))}^2\Big).
\end{align*}
Переходя обратно к переменным $x$, получаем:
\begin{equation}\label{3.4}
\|u\|_{L_2(B_{b_*}^k\setminus\om_k^\e)}^2\leqslant C\left(\e^2\eta^2\|\nabla u\|_{L_2(B_{b+1}^k\setminus\om_k^\e)}^2+\|u\|_{L_2(B^k_{b+1}\setminus B^k_{b_*})}^2\right).
\end{equation}

Дословно повторяя вывод последней оценки в доказательстве леммы 3.3 в \cite{25}, легко показать, что
\begin{align*}
\|u\|^2_{L_2(B^k_{2b_*}\setminus B^k_{b_*})}\leqslant C\Big(&\varepsilon^2\eta^2\|\nabla u\|^2_{L_2(B_{(2b+1)R_2\varepsilon}(M_k^\varepsilon)\setminus
B_{b_{*}R_2\e}(M_k^\e))}\\
&+\eta^{n}\|u\|^2_{L_2(B_{(2b+1)R_2\varepsilon}(M_k^\varepsilon)\setminus B_{(b+2)R_2\varepsilon}(M_k^\varepsilon))}\Big).
\end{align*}
Подставим последнее неравенство в оценку (\ref{3.4}) и просуммируем результат по $k\in\mathbb{M}^\e$. В результате получим:
\begin{equation}\label{3.6}
\begin{aligned}
\sum\limits_{k\in\mathbb{M}^\e}\|u\|_{L_2(B_{b_*}^k \setminus \om_{k,\e})}^2\leqslant &C\sum\limits_{k\in\mathbb{M}^\e} \Big(\e^2\eta^2\|\nabla u\|_{L_2(B_{b+1}^k\setminus\om_k^\e)}^2 \\
&+\eta^n\|u\|_{L_2(B_{(2b+1)R_2\e}(M_k^\e)\setminus B_{(b+2)R_2\e}(M_k^\e))}^2\Big).
\end{aligned}
\end{equation}

Заметим, что для $|\tau|\leqslant\tau_0$ верно равенство:
\begin{equation*}
|u(\tau,s)|^2=\int\limits_{\tau_0}^\tau\frac{\p}{\p t}\left(|u(\tau,s)|^2\chi\left(\frac{3|t|}{\tau_0}\right)\right)\,dt, \qquad \pm\tau>0,
\end{equation*}
при условии отсутствия пересечения пути интегрирования и полостей $\tht^\e$. Из последнего равенства в силу неравенства Коши-Буняковского следует:
\begin{equation}\label{3.7}
|u(\tau,s)|^2\leqslant C\left(\int\limits_{\tau_0}^\tau\left|\frac{\p u}{\p t}(\tau,s)\right|^2dt+\int\limits_{\frac{\tau_0}{3}}^{\frac{2\tau_0}{3}} |u(\tau,s)|^2\,dt\right),\qquad \pm\tau>0.
\end{equation}
Интегрируя последнее неравенство по кольцевым областям $B_{(2b+1)R_2\e}(M_k^\e) \setminus B_{(b+2)R_2\e}(M_k^\e)$
и суммируя результат по $k\in\mathbb{M}^\e$, легко получим ещё одно неравенство
\begin{equation}\label{7.18}
\sum\limits_{k\in\mathbb{M}^\e}\|u\|_{L_2(B_{(2b+1)R_2\e}(M_k^\e) \setminus B_{(b+2)R_2\e}(M_k^\e))}^2\leqslant C\e\|u\|_{W_2^1(\Om^\e)}^2.
\end{equation}
Подставляя это неравенство в (\ref{3.6}), приходим к утверждению леммы. Лемма доказана.
\end{proof}

\begin{lemma}\label{lm12.3}
Для любой функции $u\in W_2^1(\Om)$ функция $a(x,u(x))$ является элементом пространства $W_2^1(\{x:\,\dist(x,S)<\tau_0\})$ и верны оценки
\begin{equation}\label{12.20}
\begin{aligned}
&\|a(\,\cdot\,,u(x))\|_{L_2(\{x:\,\dist(x,S)<\tau_0\})}\leqslant C\|u\|_{L_2(\Om)},\\
&\|\nabla_x a(\,\cdot\,,u(x))\|_{L_2(\{x:\,\dist(x,S)<\tau_0\})}\leqslant C\|u\|_{W_2^1(\Om)},
\end{aligned}
\end{equation}
где  $C$ -- некоторая константа, не зависящая от $u$.
\end{lemma}

\begin{proof}
Из условий (\ref{2.2})
следуют неравенства
\begin{equation}\label{12.21}
|a(x,u(x))|\leqslant a_0|u(x)|,\qquad
|\nabla_x a(x,u)|\leqslant a_0|\nabla_x u(x)| + a_1 |u(x)|,
\end{equation}
из которых немедленно вытекает утверждение леммы. %Лемма доказана.
\end{proof}

Обозначим: $\tilde{S}:=\{x\in\Om:\,\tau=(2bR_2+R_0)\e\}$. Поверхность $\tilde{S}$ естественным образом параметризуем точками поверхности $S$ по следующей формуле:
\begin{equation}\label{12.36}
\tilde{x}=x+\e(2bR_2+R_0)\nu(x),
\end{equation}
где $x\in S$, $\tilde{x}\in \tilde{S}$, а $\nu$, напомним, нормаль к поверхности $S$.

\begin{lemma}\label{lm3.4}
При выполнении условий \ref{A1}, \ref{A2}, \ref{A3}, \ref{A4} для любых функций $u,v\in W_2^1(\Om^\e)$ выполнено неравенство:
\begin{equation}\label{12.15}
\begin{aligned}
\sum\limits_{k\in\mathbb{M}^\e}\bigg|\frac{\eta^{n-1}|\p\om_{k,\e}|}{|\p B_{b_*R_2}(0)|}  (a(\,\cdot\,,u),v)_{L_2(\p B_{b_*R_2\e}(M_k^\e))}&
+ (a(\,\cdot\,,u),v)_{L_2(\p\om_k^\e)}\bigg|
\\
&
\leqslant C
\e^{\frac{1}{2}}
\|u\|_{W_2^1(\Om^\e)}\|v\|_{W_2^1(\Om^\e)},
\end{aligned}
\end{equation}
где константа $C$ не зависит от параметров $k$, $\e$, $\eta$, функций $u$ и $v$.

Если дополнительно $u,v\in W_2^2(\Om\setminus \tilde{S})$,
 то оценка (\ref{12.15}) может быть улучшена:
\begin{equation}\label{7.21}
\begin{aligned}
\sum\limits_{k\in\mathbb{M}^\e}\bigg|\frac{\eta^{n-1}|\p\om_{k,\e}|}{|\p B_{b_*R_2}(0)|}  (a(\,\cdot\,,u),v&)_{L_2(\p B_{b_*R_2\e}(M_k^\e))}
+ (a(\,\cdot\,,u),v)_{L_2(\p\om_k^\e)}\bigg|
\\
&
\leqslant C
\e
\|u\|_{W_2^2(\Om\setminus \tilde{S})}\|v\|_{W_2^2(\Om\setminus \tilde{S} )},
\end{aligned}
\end{equation}
где константа $C$ не зависит от параметров $k$, $\e$, $\eta$, функций $u$ и $v$.
\end{lemma}
\begin{proof}
Всюду в доказательстве через $C$ обозначаем различные несущественные константы, не зависящие от $k$, $\e$, $\eta$, $u$ и $v$.  Обозначим:
\begin{gather*}
X_k^\e(x):=X_k\left(\frac{x-M_k^\e}{\e\eta}\right), \qquad \phi_k^\e(x):=\phi_k\left(\frac{x-M_k^\e}{\e\eta}\right),
\\
f_k^\e(x):=f_k\left(\frac{x-M_k^\e}{\e\eta}\right),
\end{gather*}
где, напомним,  функции $X_k$, $f_k$,  $\phi_k$. Верно равенство:
\begin{equation*}
\int\limits_{B_{b_*}^k\setminus\om_k^\e}a(x,u(x))\overline{v(x)}\dvr X_k^\e(x) \, dx=\e^{-1}\eta^{-1}\int\limits_{B_{b_*}^k\setminus\om_k^\e} a(x,u(x))\overline{v(x)} f_k^\e(x) \, dx.
\end{equation*}
Проинтегрируем по частям в левой части этого равенства и просуммируем по всем $k\in\mathbb{M}^\e$. В результате получим:
\begin{equation}\label{3.10}
\begin{aligned}
\sum\limits_{k\in\mathbb{M}^\e}& \Big((a(\,\cdot\,,u),v)_{L_2(\p\om_k^\e)} + (\phi_k^\e a(\,\cdot\,,u),v)_{L_2(\p B_{b_*}^k)}\Big)
\\
=&
\sum\limits_{k\in\mathbb{M}^\e} \int\limits_{B_{b_*}^k\setminus\om_k^\e}X_k^\e(x)\nabla a(x,u(x))\overline{v(x)}\,dx
\\
&+\e^{-1}\eta^{-1} \sum_{k\in\mathbb{M}^\e}\int\limits_{B_{b_*}^k\setminus\om_k^\e} a(x,u(x))\overline{v(x)} f_k^\e(x) \,dx.
\end{aligned}
\end{equation}
Возможность интегрирования по частям можно обосновать следующим образом. Вначале достаточно аппроксимировать функции $a(x,u(x))$ и $v(x)$ в норме $W_2^1(B_{b_*}^k\setminus\om_k^\e)$ бесконечно дифференцируемыми функциями и выписать приведённое равенство на основе определения (\ref{12.85}) обобщённого решения задачи (\ref{2.6}). Затем, учитывая принадлежность функций $X_k^\e$ и $f_k^\e$ пространствам $L_\infty(B_{b_*}^k\setminus\om_k^\e)$, можно уже перейти к пределу по аппроксимирующим последовательностям.

Используя леммы~\ref{lm12.3},~\ref{lm3.3}, оценим первое слагаемое в правой части равенства (\ref{3.10}):
\begin{equation}\label{3.13}
\begin{aligned}
\sum\limits_{k\in\mathbb{M}^\e}\Big|\big(X_k^\e\nabla a(\,\cdot\,,u),v_\e&\big)_{L_2(B_{b_*}^k\setminus\om_k^\e)} +\big(X_k^\e a(\,\cdot\,,u), \nabla v_\e\big)_{L_2(B_{b_*}^k\setminus\om_k^\e)}\Big|
\\
&\leqslant C\big(\e\eta +\e^{\frac{1}{2}}\eta^{\frac{n}{2}}\big) \|u\|_{W_2^1(\Om^\e)}\|v\|_{W_2^1(\Om^\e)}.
\end{aligned}
\end{equation}

Наша дальнейшая цель -- оценить второе слагаемое в правой части (\ref{3.10}). Пусть $\psi$ -- некоторая функция из пространства $W_2^1(B_{b_*}^k\setminus\om_k^\e)$. Обозначим:
\begin{equation*}
\langle \psi \rangle^k:=\frac{1}{|B_{b_*}^k\setminus\om_k^\e|} \int\limits_{B_{b_*}^k\setminus\om_k^\e}\psi\,dx,\qquad \psi^{\bot}:=\psi-\langle \psi \rangle^k.
\end{equation*}

Аналогично доказательству леммы~3.1 в \cite{25} на основе общих результатов работы \cite{Chee} устанавливается, что второе собственное значение Лапласиана с условием Неймана в области $B_{b_*R_2}(0)\setminus\om_{k,\e}$ ограничено снизу равномерно по $k$ и $\e$. Отметим ещё, что постоянная функция в этой области является собственной функцией такого оператора, соответствующего нулевому собственному значению, а также, что
\begin{equation*}
\int\limits_{B_{b_*}^k\setminus\om_k^\e} \psi^\bot\,dx=0.
\end{equation*}
Используя эти факты, аналогично рассуждениям из доказательства леммы~3.1 в \cite{25} доказывается неравенство:
\begin{equation}\label{3.14}
\|\psi^{\bot}\|_{L_2(B_{b_*}^k\setminus\om_k^\e)}\leqslant C\e\eta\|\nabla\psi\|_{L_2(B_{b_*}^k\setminus\om_k^\e)}.
\end{equation}
Оценим $\langle \psi \rangle^k$. Пользуясь неравенством Коши-Буняковского, получаем:
\begin{equation}\label{3.15}
|\langle {\psi}\rangle^k|\leqslant C\e^{-n}\eta^{-n}\int\limits_{B_{b_*}^k\setminus\om_k^\e} |\psi|\,dx\leqslant C\e^{-\frac{n}{2}}\eta^{-\frac{n}{2}} \|\psi\|_{L_2(B_{b_*}^k\setminus\om_k^\e)}.
\end{equation}
Верно равенство:
%\begin{equation}\l%abel{7.14}
\begin{align*}
\sum\limits_{k\in\mathbb{M}^\e}(a(\,\cdot\,,&u)f_k^\e, v)_{L_2(B_{b_*}^k\setminus\om_k^\e)}=\sum\limits_{k\in\mathbb{M}^\e}
\bigg(\langle a(\,\cdot\,,u)\rangle^k\left\langle{v}\right\rangle^k \int\limits_{B_{b_*}^k\setminus\om_k^\e}f_k^\e \,dx
\\
&+\langle a(\,\cdot\,,u)\rangle^k(f_k^\e, v^{\bot})_{L_2(B_{b_*}^k\setminus\om_k^\e)}
+(f_k^\e a(\,\cdot\,,u)^{\bot},v)_{L_2(B_{b_*}^k\setminus\om_k^\e)}\bigg).
\end{align*}
%\end{equation}
Оценим правую часть этого равенства. Первое слагаемое в правой части этого равенства равно нулю в силу условия (\ref{2.7}).
Используя неравенства   (\ref{3.14}), (\ref{3.15}) и леммы~\ref{lm3.3},~\ref{lm12.3}, выводим оценку:
\begin{equation*}%\l%abel{7.22}
\e^{-1}\eta^{-1}\sum\limits_{k\in\mathbb{M}^\e}\big|(a(\,\cdot\,,u)f_k^\e, v)_{L_2(B_{b_*}^k\setminus\om_k^\e)}\big|\leqslant C \big(\e\eta +\e^{\frac{1}{2}}\eta^{\frac{n}{2}}\big)
\|u\|_{W_2^1(\Om^\e)} \|v\|_{W_2^1(\Om^\e)}.
\end{equation*}
Из последней оценки, неравенства (\ref{3.13}) и равенства (\ref{3.10}) следует
\begin{align*}
\sum\limits_{k\in\mathbb{M}^\e}\big|(a(\,\cdot\,,u),v)_{L_2(\p\om_k^\e)} & + (\phi_k^\e a(\,\cdot\,,u),v)_{L_2(\p B_{b_*}^k)}\big|
\\
&\leqslant C\big(\e\eta+\e^{\frac{1}{2}}\eta^{\frac{n}{2}}\big) \|u\|_{W_2^1(\Om^\e)}\|v\|_{W_2^1(\Om^\e)}.
\end{align*}
Покажем теперь, что не ухудшая полученные оценки, функцию $\phi_k^\e$ в этих оценках можно заменить на её подходящее среднее.

Пусть $\varphi$ -- некоторая функция из пространства $W_2^1(B^k_b\setminus B^k_{b_*})$. Обозначим
\begin{align*}
\left\langle{\varphi}\right\rangle_k:=\frac{1}{|B^k_b\setminus B^k_{b_*}|}\int\limits_{B^k_b\setminus B^k_{b_*}}\varphi \,dx,\qquad \varphi_{\bot}:=\varphi-\left\langle{\varphi}\right\rangle_k.
\end{align*}
Ясно, что
\begin{equation*}
\int\limits_{B^k_b\setminus B^k_{b_*}}\varphi_{\bot} \,dx=0.
\end{equation*}
Оценим норму функции $\varphi_{\bot}$. Делая замену $\tilde{y}=(x-M_k^\e)\e^{-1}\eta^{-1}$ и применяя в растянутых переменных неравенство Пуанкаре, выводим неравенство:
\begin{align*}
&\|\varphi_{\bot}(M_k+\e\eta\,\cdot\,)\|_{L_2(\p B_{b_* R_2}(0))}\leqslant C\|\nabla_{\tilde{y}}\varphi(M_k+\e\eta\,\cdot\,) \|_{L_2(B_{bR_2}(0)\setminus B_{b_*R_2}(0))}.
\end{align*}
Переходя обратно к переменным $x$, получаем:
\begin{equation}\label{3.17}
\|\varphi_{\bot}\|_{L_2(\p B_{b_*}^k)}\leqslant C\e^{\frac{1}{2}}\eta^{\frac{1}{2}}\|\nabla\varphi\|_{L_2(B_b^k\setminus B_{b_*}^k)}.
\end{equation}
Верна оценка, аналогичная (\ref{3.15}):
\begin{align}\label{3.18}
| \langle \vp \rangle_k|\leqslant
C\e^{-\frac{n}{2}}\eta^{-\frac{n}{2}}\|\psi\|_{L_2(B^k_b\setminus B^k_{b_*})}.
\end{align}
Имеет место равенство:
\begin{equation}\label{3.19}
\begin{aligned}
\big(\phi_k^\e a(\,\cdot\,,u),v\big)_{L_2(\p B_{b_*}^k)}=& \langle a(\,\cdot\,,u)\rangle_k \langle v \rangle_k \int\limits_{\p B_{b_*}^k}\phi_k^\e\,ds
\\
&+\big(\phi_k^\e a(\,\cdot\,,u)_{\bot},v\big)_{L_2(\p B_{b_*}^k)}
\\
&+\langle a(\,\cdot\,,u)\rangle_k \big(\phi_k^\e,v_{\bot}\big)_{L_2(\p B_{b_*}^k)}.
\end{aligned}
\end{equation}

В силу выполнено
\begin{equation*}%\l%abel{3.20}
\int\limits_{\p B_{b_*}^k}\phi_k^\e\,ds=(\e\eta)^{n-1}|\p\om_{k,\e}|.
\end{equation*}
Подставляя последнее равенство в (\ref{3.19}) и суммируя результат по всем $k\in\mathbb{M}^\e$, получим:
\begin{equation}\label{3.21}
\begin{aligned}
\sum\limits_{k\in\mathbb{M}^\e}&\Big(\big(\phi_k^\e a(\,\cdot\,,u),v\big)_{L_2(\p B_{b_*}^k)}-(\e\eta)^{n-1}|\p\om_{k,\e}| \langle  a(\,\cdot\,,u)  \rangle_k \langle v  \rangle_k\Big)
\\
&=\sum_{k\in\mathbb{M}^\e}\big(\phi_k^\e (a(\,\cdot\,,u))_{\bot},v\big)_{L_2(\p B_{b_*}^k)} +\sum_{k\in\mathbb{M}^\e}\langle a(\,\cdot\,,u)\rangle_k \big(\phi_k^\e,v_{\bot}\big)_{L_2(\p B_{b_*}^k)}.
\end{aligned}
\end{equation}
Оценим правую часть последнего равенства. В силу неравенства (\ref{3.17}) и лемм~\ref{lm3.1},~\ref{lm12.3} выполнено:
\begin{equation*}%\%label{3.22}
\sum\limits_{k\in\mathbb{M}^\e}\big|\big(\phi_k^\e (a(\,\cdot\,,u))_{\bot},v\big)_{L_2(\p B_{b_*}^k)}\big|
\leqslant C(\e\eta+
\e^{\frac{1}{2}}\eta^{\frac{n}{2}})
\|u\|_{W_2^1(\Om^\e)}\|v\|_{W_2^1(\Om^\e)}.
\end{equation*}
Применяя неравенства (\ref{3.17}), (\ref{3.18}), (\ref{12.21})  и леммы~\ref{lm3.1},~\ref{lm3.3}, оценим второе слагаемое в правой части (\ref{3.21}):
\begin{equation*}%\l%abel{7.16}
\sum_{k\in\mathbb{M}^\e} \big|\langle a(\,\cdot\,,u) \rangle_k(\phi_k^\e,v_{\bot})_{L_2(\p B_{b_*}^k)}\big|\leqslant C\big(\e \eta +\e^\frac{1}{2}\eta^{\frac{n}{2}}\big) \|u\|_{W_2^1(\Om^\e)}\|v\|_{W_2^1(\Om^\e)}.
\end{equation*}
В силу последних двух неравенств и равенства (\ref{3.21}) имеем:
\begin{align*}
\sum\limits_{k\in\mathbb{M}^\e}\Big| \big(\phi_k^\e a(\,\cdot\,,u),v\big)_{L_2(\p B_{b_*}^k)}-(\e\eta)^{n-1}&
|\p\om_{k,\e}| \langle a(\,\cdot\,,u) \rangle_k \langle  v \rangle_k\Big|
\\
&
\leqslant C\big(\e \eta +\e^\frac{1}{2}\eta^{\frac{n}{2}}\big) \|u\|_{W_2^1(\Om^\e)}\|v\|_{W_2^1(\Om^\e)}.
\end{align*}
Аналогично проверяем, что
%\begin{equation}\l%abel{7.17}
\begin{align*}
\sum\limits_{k\in\mathbb{M}^\e}\bigg|\frac{|\p\om_{k,\e}|}
{|\p B_{b_*}(0)|}\big(a(\,\cdot\,,u),v&\big)_{L_2(\p B_{b_*}^k)}-(\e\eta)^{n-1}|\p\om_{k,\e}| \langle a(\,\cdot\,,u)  \rangle_k \langle  v  \rangle_k \bigg|
\\
&\leqslant C \big(\e \eta +\e^\frac{1}{2}\eta^{\frac{n}{2}}\big) \|u\|_{W_2^1(\Om^\e)}\|v\|_{W_2^1(\Om^\e)}.
\end{align*}
%\end{equation}
Из последних двух неравенств и (\ref{3.13}) выводим:
\begin{equation}\label{3.8}
\begin{aligned}
\sum\limits_{k\in\mathbb{M}^\e}\bigg|\frac{|\p\om_{k,\e}|}{|\p B_{b_*}(0)|}\big(a(\,\cdot\,,u),v&\big)_{L_2(\p B_{b_*}^k)} + \big(a(\,\cdot\,,u),v\big)_{L_2(\p\om_k^\e)}\bigg|
\\
&\leqslant C\left(\e\eta+\e^{\frac{1}{2}}\eta^{\frac{n}{2}}\right) \|u\|_{W_2^1(\Om^\e)}\|v\|_{W_2^1(\Om^\e)}.
\end{aligned}
\end{equation}
Проинтегрируем по частям в следующем интеграле:
\begin{equation}\label{12.14}
\begin{aligned}
0=&\frac{(\e\eta)^{n-1}|\p\om_{k,\e}|}{(2-n) |\p B_1(0)|}  \int\limits_{B_{b_*R_2\e}(M_k^\e)\setminus B_{b_*}^k} a(x,u)\overline{v} \D |x-M_k^\e|^{-n+2}\,dx
\\
=& \frac{\eta^{n-1}|\p\om_{k,\e}|}{|\p B_{b_*R_2}(0)|}  \big(a(\,\cdot\,,u),v\big)_{L_2(\p B_{b_*R_2\e}(M_k^\e))}
\\
&- \frac{|\p\om_{k,\e}|}{|\p B_{b_*R_2}(0)|}  \big(a(\,\cdot\,,u),v\big)_{L_2(\p B_{b_*}^k))}
 %\\
 %&
 -\frac{(\e\eta)^{n-1} |\p\om_{k,\e}|}{(2-n)|\p B_1(0)|}
 &\\
 \hphantom{-}&\cdot\int\limits_{B_{b_* R_2\e(M_k^\e)}\setminus B_{b_*}^k} \nabla |x-M_k^\e|^{-n+2}\cdot \nabla (a(x,u(x))\overline{v(x)})\,dx.
\end{aligned}
\end{equation}
Элементарные оценки, неравенство Коши-Буняковского и лемма~\ref{lm3.3} с $\eta=1$ немедленно дают:
\begin{equation}\label{7.20}
\begin{aligned}
\sum\limits_{k\in\mathbb{M}^\e}\Bigg|&\frac{(\e\eta)^{n-1} |\p\om_{k,\e}|}{(2-n)|\p B_1(0)|} \int\limits_{B_{b_* R_2\e(M_k^\e)}\setminus B_{b_*}^k} \nabla |x-M_k^\e|^{-n+2}\cdot \nabla (a(x,u(x))\overline{v(x)})\,dx
\Bigg|
\\
&\leqslant C\sum\limits_{k\in\mathbb{M}^\e} \Big(\|v\|_{L_2(B_{b_*R_2\e(M_k^\e)}\setminus B_{b_*}^k)} \|\nabla a(\,\cdot\,,u)\|_{L_2(B_{b_*R_2\e(M_k^\e)}\setminus B_{b_*}^k)}
\\
 &\hphantom{\leqslant C\sum\limits_{k\in\mathbb{M}^\e} \Big(} +\|\nabla v\|_{L_2(B_{b_*R_2\e(M_k^\e)}\setminus B_{b_*}^k)} \|  a(\,\cdot\,,u)\|_{L_2(B_{b_*R_2\e(M_k^\e)}\setminus B_{b_*}^k)}
\Big)
\\
&\leqslant C\e^{\frac{1}{2}} \|u\|_{W_2^1(\Om^\e)} \|v\|_{W_2^1(\Om^\e)}.
\end{aligned}
\end{equation}
Учитывая данную оценку, выразим теперь скалярное произведение $(a(\,\cdot\,,u),v)_{L_2(\p B_{b_*}^k)}$
из равенства (\ref{12.14}) и подставим полученное выражение в (\ref{3.8}). Тогда немедленно получим требуемую оценку (\ref{12.15}).

Если $u,v\in W_2^2(\Om\setminus S)$, то все приведённые выше оценки могут быть улучшены за счёт дополнительного применения неравенств, вытекающих из леммы~\ref{lm3.3} и оценок (\ref{12.21}):
 \begin{equation}\label{7.23}
\begin{aligned}
&\sum\limits_{k\in\mathds{M}^\e} \|\nabla a(\,\cdot\,, u)\|_{L_2(B_b^k\setminus B_{b_*}^k)}^2\leqslant
C \sum\limits_{k\in\mathds{M}^\e} \|u\|_{W_2^1(B_b^k\setminus B_{b_*}^k)}^2
\\
&\hphantom{\sum\limits_{k\in\mathds{M}^\e} \|\nabla a(\,\cdot\,, u)\|_{L_2(B_b^k\setminus B_{b_*}^k)}^2} \leqslant C(\e^2\eta^2+\e\eta^n)\|u\|_{W_2^2(\Om\setminus \tilde{S})}^2,
\\
&\sum\limits_{k\in\mathds{M}^\e} \|\nabla v\|_{L_2(B_b^k\setminus B_{b_*}^k)}^2\leqslant C(\e^2\eta^2+\e\eta^n)\|v\|_{W_2^2(\Om\setminus \tilde{S})}^2,
\\
&\sum\limits_{k\in\mathbb{M}^\e} \|\nabla a(\,\cdot\,,u)\|_{L_2(B_{b_*R_2\e(M_k^\e)\setminus B_{b_*}^k})}^2
\leqslant C\sum\limits_{k\in\mathbb{M}^\e} \| u\|_{W_2^1(B_{b_*R_2\e(M_k^\e)}\setminus B_{b_*}^k)}^2
\\
&\hphantom{\sum\limits_{k\in\mathbb{M}^\e} \|\nabla a(\,\cdot\,,u)\|_{L_2(B_{b_*R_2\e(M_k^\e)\setminus B_{b_*}^k})}^2}\leqslant C\e \|u\|_{W_2^2(\Om\setminus \tilde{S})}^2,
\\
&\sum\limits_{k\in\mathbb{M}^\e} \|\nabla a(\,\cdot\,,v)\|_{L_2(B_{b_*R_2\e(M_k^\e)\setminus B_{b_*}^k})}^2
\leqslant C\e \|v\|_{W_2^2(\Om\setminus \tilde{S})}^2.
\end{aligned}
 \end{equation}
Это улучшение приводит к замене выражений  $\e\eta+\e^\frac{1}{2}\eta^{\frac{n}{2}}$ на $\e^2\eta^2+\e\eta^n$ в привёденных выше оценках, а в оценке (\ref{7.20}) степень $\e^\frac{1}{2}$ заменяется на $\e$. В результате мы приходим к неравенству (\ref{7.21}).
Лемма доказана.
\end{proof}

\begin{lemma}\label{lm3.5}
Для любой функции  $v\in W_2^1(\Omega_\e)$ выполнено неравенство:
\begin{equation*}
\|v\|_{L_2(\tilde{S})}^2\leqslant\delta\|\nabla v\|^2_{L_2(\Om^\e)}+C(\delta)\|v\|^2_{L_2(\Om^\e)},
\end{equation*}
где $\delta>0$ -- некоторая константа, константа $C(\delta)>0$ не зависит от $v$.
\end{lemma}
Доказательство леммы проводится аналогично доказательству леммы 3.1 из \cite{24}.

Функции $\a^\e$ и $\a$, заданные на $S$, определим также и на поверхности $\tilde{S}$ с помощью параметризации (\ref{12.36}) по следующему правилу:
\begin{equation}\label{12.1}
\a^\e(\tilde{x}):=\a^\e(x), \qquad \a(\tilde{x}):=\a(x),
\end{equation}
где точки $\tilde{x}\in\tilde{S}$ и $x\in S$ связаны формулой (\ref{12.36}).
Напомним, что в силу условия~\ref{A5} функция $\a$ является элементом пространства $W_\infty^1(S)$. Поэтому продолжение этой функции, введённое в (\ref{12.1}), является и элементом пространства $W_\infty^1(\tilde{S})$.

Для произвольной функции обозначим
\begin{equation*}
[u]_{\tilde{S}}:=u\big|_{\tau=(2bR_2+R_0)\e+0}-u\big|_{\tau=(2bR_2+R_0)\e-0}
\end{equation*}
и рассмотрим краевую задачу:
\begin{equation}\label{3.24}
\begin{gathered}
\bigg(-\sum\limits_{i,j=1}^n\frac{\p}{\p x_i} A_{ij}\frac{\p}{\p x_j}+\sum\limits_{j=1}^n
A_j\frac{\p}{\p x_j}
+A_0-\lambda \bigg)\tilde{u}_0=f\quad\text{в}\quad
\Om,
\\
\tilde{u}_0=0\quad\text{на}\quad \p\Om,\qquad[\tilde{u}_0]_{\tilde{S}}=0,\qquad \left[\frac{\p\tilde{u}_0}{\p\mathrm{n}}\right]_{\tilde{S}}+\a^0 a(\,\cdot\,,\tilde{u}_0)|_{\tilde{S}}=0.
\end{gathered}
\end{equation}

\begin{lemma}\label{lm3.7}
Существует фиксированное $\l_0$, не зависящее от $\e$, такое что
при $\l<\l_0$ задачи (\ref{2.13}), (\ref{2.15}) и (\ref{3.24}) однозначно разрешимы для любой $f\in L_2(\Om)$ и выполнены неравенства:
\begin{equation}\label{3.25}
\begin{aligned}
&\|\tilde{u}_0-u_0\|_{W_2^2(\Om\setminus(S\cup\tilde{S}))}\leqslant C \e^{\frac{1}{2}}(\|\a^0\|_{W_{\infty}^1(S)}+1)\|f\|_{L_2(\Om)},
\\
&\|\tilde{u}_0\|_{W_2^2(\Om\setminus\tilde{S})}\leqslant C(\|\a^0\|_{W_{\infty}^1(S)}+1)\|f\|_{L_2(\Om)},
\end{aligned}
\end{equation}
где константа $C$ не зависит от $\e$, $\a^0$ и $f$.
\end{lemma}

Существование $\l_0$ и разрешимость задач (\ref{2.13}), (\ref{2.15}) и (\ref{3.24}) легко проверяется аналогично доказательству леммы~\ref{lm3.2}. Проверка оценок (\ref{3.25})  основано на применении леммы 8.1 из \cite[Гл.3, \S 8]{27} и дословно воспроизводит доказательство леммы 3.7 из \cite{24}, где оно было дано для случая $n=2$.  При этом размерность области $\Om$ не играет никакой роли в доказательстве леммы.

Обозначим: $\tilde{\Om}^\e:=\big\{x\in\Om^\e:\, (2bR_2+R_0)\e<\tau<\frac{\tau_0}{2}\big\}$.

\begin{lemma}\label{lm12.2}
Пусть $\a\in L_\infty(S)$ -- произвольная функция, которую продолжим на поверхность $\tilde{S}$ согласно (\ref{12.1}). Тогда для всех $u,v\in W_2^1(\tilde{\Om}^\e)$ верна оценка:
\begin{equation*}%\l%abel{12.5}
(\a u,v)_{L_2(\tilde{S})} \leqslant C(\|\a\|_S+\e) \|u\|_{W_2^1(\tilde{\Om}^\e)} \|v\|_{W_2^1(\tilde{\Om}^\e)},
\end{equation*}
 где $C$ -- некоторая константа, не зависящая от параметра $\e$ и функций $u$, $v$.
\end{lemma}

\begin{proof}
Функции $u$, $v$ продолжим в область $\vw\setminus\tilde{\Om}^\e$ чётным образом относительно $\tilde{S}$. А именно, для каждой точки $x\in\vw\setminus\tilde{\Om}^\e$ однозначно найдём точки $s\in S$ и $\tau\in(0,(2bR_2+R_0)\e)$ по правилу $x=s+\tau \nu(s)$ и положим
\begin{equation*}
u(x)=u\big(s+((4bR_2+2R_0)\e-\tau)\nu(s)\big),\quad v(x)=v\big(s+((4bR_2+2R_0)\e-\tau)\nu(s)\big).
\end{equation*}
В силу условия~\ref{A1} такое продолжение определено корректно, продолженные функции являются элементами пространства $W_2^1(\vw)$ и верны оценки
\begin{equation}\label{12.38}
\|u\|_{W_2^1(\vw)}\leqslant C\|u\|_{W_2^1(\tilde{\Om}^\e)}, \qquad
\|v\|_{W_2^1(\vw)}\leqslant C\|v\|_{W_2^1(\tilde{\Om}^\e)}.
\end{equation}
Здесь и всюду до конца доказательства через  $C$ обозначаем различные константы, не зависящие от  $\e$, $u$, $v$. Отметим ещё, что в силу равенства
\begin{equation*}
u\big|_{\tau=(2bR_2+R_0)\e}=u\big|_{\tau=0} + \int\limits_{0}^{(2bR_2+R_0)\e}
\frac{\p u}{\p\tau}\,d\tau
\end{equation*}
верна оценка
\begin{equation}\label{12.39}
\Big\|u\big|_{\tau=(2bR_2+R_0)\e}-u\big|_{\tau=0}\big\|_{L_2(S)}\leqslant C\e \|u\|_{W_2^1(\vw)}.
\end{equation}
Такая же оценка верна и для функции $v$. Отметим ещё, что дифференциалы площади повехностей $S$ и $\tilde{S}$ связаны равенствами $d\tilde{s}=(1+\e J_\e(s))ds$, где $J_\e$ -- непрерывно дифференцируемая функция, ограниченная равномерно по $\e$ и $s\in S$ вместе со  своими пространственными производными первого порядка.

Используя указанные свойства дифференциалов площадей $S$ и $\tilde{S}$ и оценки (\ref{12.39}), (\ref{12.24}), (\ref{12.38}), получаем:
\begin{equation*}%\l%abel{12.40}
\big|(\a u,v)_{L_2(\tilde{S})}-(\a u,v)_{L_2(S)}\big| \leqslant C\e \|u\|_{W_2^1(\tilde{\Om}^\e)}  \|v\|_{W_2^1(\tilde{\Om}^\e)}.
\end{equation*}
Применяя теперь к скалярному произведению $(\a u,v)_{L_2(S)}$ вторую оценку из (\ref{12.33}), приходим к утверждению леммы. Лемма доказана.
 \end{proof}

\begin{lemma}\label{lm3.6} Пусть выполнено условие~\ref{A5}. Тогда для всех
$v\in W_2^1(\Om^\e)$ верна оценка
\begin{equation*}%\l%abel{3.23}
\big|\big((\a^\e-\a^0)a(\,\cdot\,,\tilde{u}_0),v_\e\big)_{L_2(\tilde{S})}\big|
\leqslant C(\kappa(\e)+\e) \|\tilde{u}_0\|_{W_2^1(\Om)} \|v_\e\|_{W_2^1(\Om^\e)},
\end{equation*}
где константа $C$ не зависит от $\e$, $\tilde{u}_0$ и $v$.
\end{lemma}
\begin{proof}
Так как $\tilde{u}_0\in W_2^1(\Om)$, то в силу леммы~\ref{lm12.3} функция $a(x,\tilde{u}_0)$ является элементом пространства $W_2^1(\tilde{\Om}^\e)$ и верна оценка
\begin{equation*}%\l%abel{12.13}
\|a(\,\cdot\,,\tilde{u}_0)\|_{W_2^1(\tilde{\Om}^\e)} \leqslant C\|\tilde{u}_0\|_{W_2^1(\Om)},
\end{equation*}
где $C$ -- некоторая константа, не зависящая от $\e$ и $\tilde{u}_0$.
 Применяя теперь лемму~\ref{lm12.2} с $u=a(x,\tilde{u}_0)$ и $\a=\a^\e-\a^0$ и учитывая условие~\ref{A5}, приходим к требуемой оценке. Лемма доказана.
 \end{proof}

\section{Усреднённая задача без условий на $S$}

В настоящем параграфе мы доказываем теорему~\ref{th1}. Всюду в доказательстве считаем, что параметр $\l$ выбирается из условия $\l<\l_0$, где $\l_0$ -- отрицательное и достаточно большое по модулю число так, что оно не превосходит аналогичную константу из леммы~\ref{lm3.2}.

Разность решений задач задач (\ref{2.5}) и (\ref{2.12}), обозначаемая через $v_\e=u_\e-u_0$, удовлетворяет краевой задаче
\begin{equation}\label{7.7}
\begin{gathered}
\bigg(-\sum\limits_{i,j=1}^n\frac{\p}{\p x_i} A_{ij}\frac{\p}{\p x_j}+\sum\limits_{j=1}^n
A_j\frac{\p}{\p x_j}
+A_0-\lambda\bigg) v_\e=0\quad\text{в}\quad\Om^\e,
\\
v_\e=0\quad\text{на}\quad\p\Om,\qquad
\frac{\p v_\e}{\p \mathrm{n}}
=-\frac{\p u_0}{\p \mathrm{n}}
-a(\,\cdot\,,u_\e) \quad\text{на}\quad\p\tht^\e.
\end{gathered}
\end{equation}
Выпишем для этой задачи интегральное тождество, взяв $v_\e\in\Ho^1(\Om^\e,\p\Om)$ в качестве пробной функции:
\begin{equation}\label{4.1}
\mathfrak{h}_0(v_\e,v_\e)-\l \|v_\e\|_{L_2(\Om^\e)}^2=
-\left(\frac{\p u_0}{\p \mathrm{n}},v_\e\right)_{L_2(\p\tht^\e)}
-\big(a(\,\cdot\,,u_\e),v_\e\big)_{L_2(\p\tht^\e)}.
\end{equation}

Основная идея доказательства теоремы состоит в том, чтобы оценить сверху правую часть равенства (\ref{4.1}) и снизу   левую часть этого равенства, что в итоге даст оценку для функции $v_\e$.

Вначале рассмотрим случай $a\equiv0$. В этом случае второе слагаемое в правой части равенства (\ref{4.1}) равняется  нулю, а для оценки
 первого слагаемого  проинтегрируем по частям следующим образом:
\begin{equation}\label{4.0}
 \begin{aligned}
\int\limits_{B_1^k\setminus \om_k^\e}\overline{v}_\e \bigg(&-\sum\limits_{i,j=1}^n\frac{\p}{\p x_i}A_{ij}\frac{\p}{\p
x_j}+\sum\limits_{j=1}^n
A_j\frac{\p}{\p x_j}
+A_0\bigg)u_0 \,dx
\\
=&\int\limits_{\p\om_k^\e}\frac{\p u_0}{\p \mathrm{n}}\overline{v}_\e\,ds-\int\limits_{\p B_1^k}\frac{\p u_0}{\p \mathrm{n}}\overline{v}_\e \,ds+\sum\limits_{i,j=1}^n\int\limits_{B_1^k\setminus \om_k^\e} A_{ij}\frac{\p u_0}{\p
x_j}\frac{\p\overline{v}_\e}{\p x_i}\,dx
\\
&+\sum\limits_{j=1}^n\int\limits_{B_1^k\setminus \om_k^\e} A_j\frac{\p
u_0}{\p x_j}\overline{v}_\e \,dx
+\int\limits_{B_1^k\setminus \om_k^\e} A_0 u_0\overline{v}_\e\,dx.
\end{aligned}
\end{equation}
Из последнего равенства и уравнения из (\ref{2.12}) следует
\begin{equation}\label{4.2}
\begin{aligned}
\left(\frac{\p u_0}{\p \mathrm{n}},v_\e \right)_{L_2(\p \om_k^\e)}=&\left(\frac{\p u_0}{\p \mathrm{n}},v_\e \right)_{L_2(\p B^k_1)}-\sum\limits_{i,j=1}^n\left(A_{ij}\frac{\p u_0}{\p
x_j},\frac{\p v_\e}{\p  x_i}\right)_{L_2( B_1^k\setminus \om_k^\e)}
\\
&-\sum\limits_{j=1}^n\left(A_j\frac{\p u_0}{\p x_j},v_\e \right)_{L_2(B_1^k\setminus \om_k^\e)}
-(A_0u_0,v_\e)_{L_2(B_1^k\setminus \om_k^\e)}
\\
&+(f,v_\e)_{L_2(B_1^k\setminus \om_k^\e)}
+\lambda (u_0,v_\e)_{L_2( B_1^k\setminus \om_k^\e)}.
\end{aligned}
\end{equation}
Введем вспомогательную задачу
\begin{equation}\label{4.3}
\begin{gathered}
\Delta W_{k,i}^\e=0\quad\text{в}\quad B_{b_*}^k\setminus B_1^k,
\\
\frac{\p W_{k,i}^\e}{\p r}=\frac{\p\vr_i^k}{\p r}\quad\text{на}\quad\p B_1^k,\quad \frac{\p W_{k,i}^\e}{\p r}=0\quad\text{на}\quad\p B_{b_{*}}^k,
\end{gathered}
\end{equation}
где $\vr^k=(\vr^k_1,\ldots,\vr^k_n)=x-M_k^\e$, $r=|\vr^k|$. Решением этой задачи является функция
\begin{align*}
W_{k,i}^\e=\frac{-(b+1)^{-n}\vr_i^k}{2^{-n}-(b+1)^{-n}} +\frac{2^{-n}r^{-n}\vr_i^k}{(-n+1)(R_2\eta\e)^{-n}(2^{-n}-(b+1)^{-n})}.
\end{align*}
Эта функция удовлетворяет неравенству:
\begin{equation}\label{4.4}
|\nabla W_{k,i}^\e|\leqslant C\quad \text{в}\quad B_{b_{*}}^k\setminus B_1^k,
\end{equation}
где константа $C$ не зависит от  $W_{k,i}^\e$. Проинтегрируем по частям в равенстве
\begin{equation*}%\l%abel{4.10}
\sum_{i,j=1}^n\int\limits_{B_{b_{*}}^k\setminus B_1^k} A_{ij}\frac{\p u_0}{\p x_j}\overline{v}_\e\Delta W_{k,i}^\e \,dx=0
\end{equation*}
с учётом граничных условий задачи (\ref{4.3}). В результате получим:
\begin{equation*}%\l%abel{4.5}
\left(\frac{\p u_0}{\p \mathrm{n}}, v_\e\right)_{L_2(\p B_1^k)}=\sum_{i,j=1}^n\int\limits_{B_{b_{*}}^k\setminus B_1^k} \nabla W_{k,i}^\e
\nabla A_{ij}\frac{\p u_0}{\p x_j}\overline{v}_\e \,dx.
\end{equation*}
Из последнего равенства, (\ref{4.2}) и (\ref{4.4}) выводим:
%\begin{equation}\l%abel{4.6}
\begin{align*}
\left|\left(\frac{\p u_0}{\p \mathrm{n}},v_\e\right)_{L_2(\p \tht^\e)}\right|\leqslant & C\bigg(\sum_{k\in
\mathbb{M}^\e}\|u_0\|^2_{W_2^1(B_{b_*}^k\setminus \om_k^\e)}\bigg)^{\frac{1}{2}}\bigg(\sum_{k\in \mathbb{M}^\e}\|\nabla
v_\e\|^2_{L_2(B_{b_*}^k\setminus \om_k^\e)}\bigg)^{\frac{1}{2}}\\
&+C\bigg(\sum_{k\in \mathbb{M}^\e}\|u_0\|^2_{W_2^2(B_{b_*}^k\setminus \om_k^\e)}\bigg)^{\frac{1}{2}}\bigg(\sum_{k\in \mathbb{M}^\e}\|
v_\e\|^2_{L_2(B_{b_*}^k\setminus \om_k^\e)}\bigg)^{\frac{1}{2}}.
\end{align*}
%\end{equation}
Здесь и всюду далее символом $C$ обозначаем константы, не зависящие $u_0$, $v_\e$ и $\e$. Правую часть последнего неравенства оценим с помощью леммы \ref{lm3.3} и оценки (\ref{4.7}):
\begin{equation}\label{4.8}
\left|\left(\frac{\p u_0}{\p \mathrm{n}},v_\e\right)_{L_2(\p \tht^\e)}\right|\leqslant C(\e\eta+\varepsilon^{\frac{1}{2}}\eta^{\frac{n}{2}}) \|f\|_{L_2(\Om)}\|v_\e\|_{W_2^1(\Om^\e)}.
\end{equation}
Из последнего неравенства и (\ref{4.9}) уже вытекает оценка (\ref{2.17}).

Теперь рассмотрим случай $a\neq0$.
Оценим правую часть равенства (\ref{4.1}). Для первого слагаемого остается справедливой оценка (\ref{4.8}).  В силу условий (\ref{2.2}) выполнено неравенство
$|a(x,u_\e)|\leqslant C|u_\e|$, применяя
которое, (\ref{3.23}) и лемму \ref{lm3.1}, приходим к оценке
\begin{equation*}
\big|\big(a(\,\cdot\,,u_\e),v_\e\big)_{L_2(\p\tht^\e)}
\big|\leqslant
C(\e\eta +\eta^{n-1})\|f\|_{L_2(\Om)}\|v_\e\|_{W_2^1(\Om^\e)}.
\end{equation*}
Неравенство (\ref{2.18}) вытекает из последней оценки, (\ref{4.8}), (\ref{4.1}) и (\ref{4.9}). Теорема \ref{th1} доказана.

\section{Усреднённая задача с дельта-взаимодействием}
\label{s5}

В данном параграфе мы доказываем теорему~\ref{th2}. По сравнению с доказательством предыдущей теоремы здесь возникают дополнительные трудности, что требует привлечения новой техники.

Первая трудность связана с тем, что многообразие $S$ может пересекать полости $\om_k^\e$ и это вызывает сложности при попытки прямого вывода нормы разности $u_\e-u_0$ по аналогии с предыдущим параграфом. Для преодоления этой трудности мы вводим многообразие $\tilde{S}$ и рассматриваем краевую задачу (\ref{3.24}). Многообразие $\tilde{S}$ не пересекает полостей $\om_k^\e$ и это в итоге позволит нам оценить разность $u_\e-\tilde{u}_0$. Поэтому вначале мы оценим норму разности $u_\e-\tilde{u}_0$, а затем уже норму разности $\tilde{u}_0-u_0$.

Как и в доказательстве теоремы~\ref{th1}, выберем и зафиксируем достаточно большое по модулю отрицательное  $\l_0$ так, чтобы гарантировать разрешимость задач для $u_\e$, $u_0$, $\tilde{u}_0$. Такая возможность гарантируется леммами~\ref{lm3.2},~\ref{lm3.7}.

Обозначим $v_\e:=u_\e-\tilde{u}_0$. Функция $v_\e$ является решением следующей задачи
%\begin{equation}\%label{7.8}
\begin{align*}
&\bigg(-\sum\limits_{i,j=1}^n\frac{\p}{\p x_i} A_{ij}\frac{\p}{\p x_j}+\sum\limits_{j=1}^n
A_j\frac{\p}{\p x_j} +A_0-\lambda \bigg)v_\e=0\quad\text{в}\quad
\Om^\e,
\\
&v_\e=0\quad\text{на}\quad \p\Om,
\qquad
\frac{\p v_\e}{\p \mathrm{n}} =-\frac{\p\tilde{u}_0}{\p\mathrm{n}} -a(\,\cdot\,,u_\e)\quad\text{на}\quad\p\tht^\e,
\\
&[v_\e]_{\tilde{S}}=0,\qquad
\left[\frac{\p v_\e}{\p\mathrm{n}}\right]_{\tilde{S}}=\a^0 a(\,\cdot\,,\tilde{u}_0)|_{\tilde{S}}.
\end{align*}
%\end{equation}
Выпишем для этой задачи интегральное тождество, взяв $v_\e$ в качестве пробной функции:
\begin{equation}\label{5.1}
\begin{aligned}
\mathfrak{h}_0(v_\e,v_\e)&+\big(a(\,\cdot\,,u_\e) -a(\,\cdot\,,\tilde{u}_0),v_\e\big)_{L_2(\p\tht^\e)} -\l\|v_\e\|_{L_2(\Om^\e)}^2
\\
=&-\left(\frac{\p\tilde{u}_0}{\p\mathrm{n}},v_\e\right)_{L_2(\p\tht^\e)}
-\big(a(\,\cdot\,,\tilde{u}_0),v_\e\big)_{L_2(\p\tht^\e)}
\\
& - (\a^0 a(\,\cdot\,,\tilde{u}_0),v_\e)_{L_2(\tilde{S})}.
\end{aligned}
\end{equation}

Наша дальнейшая цель -- оценить сверху правую часть равенства (\ref{5.1}) и снизу левую часть этого равенства.
Всюду далее до конца доказательства через $C$ обозначаем  различные несущественные константы, не зависящие от $\e$, $v_\e$, $f$, $\tilde{u}_0$, а также пространственных переменных и  индекса $k\in\mathbb{M}^\e$, который будет введён ниже.

Используя свойство (\ref{3.2}) и неравенство Коши-Буняковского, выводим:
\begin{equation*}
\Big|\big(a(\,\cdot\,,u_\e)-a(\,\cdot\,,\tilde{u}_0),v_\e\big)_{L_2(\p\tht^\e)}
\Big|
\leqslant a_0 \|u_\e-\tilde{u}_0\|_{L_2(\p\tht^\e)}\|v_\e\|_{L_2(\p\tht^\e)}\leqslant C\|v_\e\|^2_{L_2(\p\tht^\e)}.
\end{equation*}
В силу последнего неравенства,  (\ref{4.9}) и леммы~\ref{lm3.1} теперь  следует, что увеличивая при необходимости модуль  числа $\l_0$,
при $\l<\l_0$ будем иметь:
\begin{equation}\label{5.3}
\bigg| \mathfrak{h}_0(v_\e,v_\e)+\big(a(\,\cdot\,,u_\e) -a(\,\cdot\,,\tilde{u}_0),v_\e\big)_{L_2(\p\tht^\e)} -\l\|v_\e\|_{L_2(\Om^\e)}^2 \bigg|\geqslant
C \|v_\e\|_{W_2^1(\Om_\e)}^2.
\end{equation}

Первое слагаемое в правой части  неравенства (\ref{5.1}) оценивается так же, как и первое слагаемое в правой части (\ref{4.1}) в случае $a\equiv0$. Поэтому, повторяя выкладки, проведенные при выводе оценки (\ref{4.8}), получим неравенство:
\begin{equation}\label{5.2}
\begin{aligned}
\left|\left(\frac{\p\tilde{u}_0}{\p\mathrm{n}},v_\e\right)_{L_2(\p\tht^\e)}\right|\leqslant C(\e\eta+\e^{\frac{1}{2}}\eta^{\frac{n}{2}})\|f\|_{L_2(\Om)}\|v_\e\|_{W_2^1(\Om^\e)}.
\end{aligned}
\end{equation}

В дальнейших оценках, не оговаривая отдельно, мы неоднократно будем пользоваться равномерной ограниченность площадей $|\p\om_{k,\e}|$, установленной в лемме~\ref{lm12.9}. Согласно лемме \ref{lm3.4} и второй оценке в (\ref{3.25}), выполнено неравенство:
\begin{equation}\label{5.4}
\begin{aligned}
\sum\limits_{k\in\mathbb{M}^\e}&\bigg|\frac{\eta^{n-1}|\p\om_{k,\e}|} {|\p B_{b_*R_2}(0)|}(a(\,\cdot\,,\tilde{u}_0),v_\e)_{L_2(\p B_{b_*R_2\e}(M_k^\e))} + (a(\,\cdot\,,\tilde{u}_0),v_\e)_{L_2(\p\om_k^\e)}\bigg|\\
&\leqslant C  \e^{\frac{1}{2}}  \|\tilde{u}_0\|_{W_2^1(\Om)}
\|f\|_{L_2(\Om)}\|v_\e\|_{W_2^1(\Om^\e)}
\leqslant C \e^{\frac{1}{2}}
\|f\|_{L_2(\Om)}\|v_\e\|_{W_2^1(\Om^\e)}.
\end{aligned}
\end{equation}

Пусть $\xi=(\xi',\xi_n)$, $\xi'=(\xi_1,\xi_2,\ldots,\xi_{n-1})$ -- декартовы координаты в $\mathds{R}^n$,
$$
\Xi:=\{\xi:\, |\xi'|<bR_2,\, |\xi_n|<b R_2\},\qquad \Ups:= \{\xi:\, |\xi'|<bR_2, \, \xi_n=bR_2\}.
$$
Рассмотрим вспомогательную задачу:
\begin{align*}
&\Delta Y=0\quad\text{в}\quad\Xi\setminus B_{b_* R_2}(0), \qquad \frac{\p Y}{\p|\xi|}=1\quad\text{на}\quad\p B_{b_*R_2}(0),
\\
&\frac{\p Y}{\p \nu}= |\p B_{b_*}(0)| \z\left(\frac{\xi'}{R_2}\right) \qquad\text{на}\quad \Ups,\qquad \frac{\p Y}{\p \nu}=0 \quad\text{на}\quad\p\Xi\setminus\Ups.
\end{align*}
где $\nu$ -- внешняя нормаль к $\p \Xi$.
Функция $\z$ по предположению гладкая, а равенство (\ref{2.8b}) обеспечивает выполнение условия разрешимости этой задачи:
\begin{equation*}
\int\limits_{\p B_{b_*R_2}(0)}\,d\xi=\int\limits_{\Ups}|\p B_{b_*}(0)| \z\left(\frac{\xi'}{R_2}\right)\, d\xi'=|\p B_{b_*}(0)|R_2^{n-1}\int\limits_{\mathds{R}^{d-1}} \z(\xi')\, d\xi'.
\end{equation*}
Существует единственное решение этой задачи, удовлетворяющее условию
\begin{equation*}
\int\limits_{\Xi\setminus B_{b_* R_2}(0)} Y(\xi)\,d\xi=0.
\end{equation*}
Далее считаем, что функция $Y$ выбрана из этого условия. Кроме того, в силу стандартных теорем повышения гладкости сразу заключаем, что функция $Y$ по крайней мере является элементом пространства $W_\infty^1(\Xi\setminus B_{b_* R_2}(0))$.

Фиксируем теперь произвольный индекс $k\in \mathbb{M}^\e$ определим переменные $\xi$  следующим образом: $\xi:=y\e^{-1}$, где $y=(y_1,\ldots,y_n)$ -- декартовы координаты в $\mathds{R}^n$ с центром в точке $M_k^\e$, причём ось $y_n$ направлена вдоль положительного направления вектора нормали к поверхности $S$ в точке $M_{k,\bot}^\e$.
Соответствующую функцию $Y(\xi)$, выраженную таким образом через переменные $(s,\tau)$ и, следовательно, через переменные $x$, обозначим символом $Y^\e(x)$. Ещё положим:
$\Xi_k^\e:=\{x:\, \xi\in \Xi\}$, $\Ups_k^\e:=\{x:\, \xi\in\Ups\}$.
Проинтегрируем по частям в равенстве
\begin{equation*}
\e\int\limits_{\Xi_k^\e\setminus B_{b_*R_2\e}(M_k^\e)}a(\,\cdot\,,\tilde{u}_0)\overline{v}_\e\Delta Y^\e \,d\xi=0.
\end{equation*}
В результате получим
\begin{align*}
 |\p B_{b_*}(0)| \left(\z\left(\frac{|\xi'|}{R_2}\right)
 a(\,\cdot\,,\tilde{u}_0),v_\e\right)_{L_2(\Ups_k^\e)} & -(a(\,\cdot\,,\tilde{u}_0),v_\e)_{L_2(\p B_{b_*R_2\e}(M_k^\e))}
 \\
 &=\e\int\limits_{\Xi\setminus B_{b_*R_2\e}(M_k^\e)}\nabla a(\,\cdot\,,\tilde{u}_0)\overline{v}_\e \nabla Y^\e\,dx.
\end{align*}
Суммируя последние равенства по $k\in\mathbb{M}^\e$ и учитывая неравенства  (\ref{12.20}), (\ref{5.4}), лемму \ref{lm3.4} и второе неравенство в (\ref{3.25}), выводим:
\begin{equation}\label{5.5}
\begin{aligned}
\bigg|(a(\,\cdot\,,\tilde{u}_0),v_\e)_{L_2(\p \theta^\e)}& +  \sum\limits_{k\in\mathbb{M}^\e} \frac{\eta^{n-1}|\p \om_{k,\e}|}{R_2^{n-1}}   \left(\z\left(\frac{|\xi'|}{R_2}\right)
 a(\,\cdot\,,\tilde{u}_0),v_\e\right)_{L_2(\Ups_k^\e)}\bigg|
 \\
 &\leqslant C
\varepsilon^{\frac{1}{2}}
\|f\|_{L_2(\Omega)}\|v_\e\|_{W_2^1(\Omega_\e)}.
\end{aligned}
\end{equation}

Определим множества
$$
\Om_k^\e:=\big\{x\in\Om:\, |\xi'|<bR_2,\, \xi_n>bR_2,\, \tau<\e(2bR_2+R_0)\big\}.
$$
Это цилиндрические области, нижними основаниями которых служат $\Ups_k^\e$, а верхними -- пересечения $\tilde{S}\cap B_{bR_2\e}(\tilde{M}_{k,\bot}^\e)$, где $\tilde{M}_{k,\bot}^\e$ -- точка пересечения оси $Oy_n$ с поверхностью $\tilde{S}$. С учётом финитности срезающей  функции $\z$ проинтегрируем по частям следующим образом:
\begin{equation}
\begin{aligned}
\int\limits_{\Om_k^\e} \z\left(\frac{|\xi'|}{R_2}\right) &\frac{\p\ }{\p y_n} a(\,\cdot\,,\tilde{u}_0)\overline{v}_\e\,dx=-\left( \z\left(\frac{|\xi'|}{R_2}\right)a(\,\cdot\,,\tilde{u}_0), v_\e\right)_{L_2(\Ups_k^\e)} \\
&+\left( \z\left(\frac{|\xi'|}{R_2}\right)a(\,\cdot\,,\tilde{u}_0),\cos(Oy_n,\tilde{\nu}) v_\e\right)_{L_2(\tilde{S}\cap B_{bR_2\e}(\tilde{M}_{k,\bot}^\e))},
\end{aligned}\label{12.18}
\end{equation}
где $\tilde{\nu}$ -- нормаль к поверхности $\tilde{S}$, направленная от поверхности $S$. Ясно, что верна равномерная по $\e$, $k$ и $x\in \tilde{S}\cap B_{bR_2\e}(\tilde{M}_{k,\bot}^\e)$ оценка
\begin{equation}\label{12.17}
\big|\cos(Oy_n,\tilde{\nu})-1\big|\leqslant C\e.
\end{equation}
Из данной оценки, равенства (\ref{12.18}) и интегрирования оценки (\ref{3.7}) по соответствующим областям следует, что
\begin{equation}\label{6.7a}
\begin{aligned}
 \bigg|&\sum\limits_{k\in\mathbb{M}^\e}
\frac{\eta^{n-1}|\p\om_{k,\e}|}{R_2^{n-1}}
\left( \z\left(\frac{|\xi'|}{R_2}\right)a(\,\cdot\,,\tilde{u}_0), v_\e\right)_{L_2(\Ups_k^\e)}
\\
&- \sum\limits_{k\in\mathbb{M}^\e}
\frac{\eta^{n-1}|\p\om_{k,\e}|}{R_2^{n-1}}  \left( \z\left(\frac{|\xi'|}{R_2}\right)a(\,\cdot\,,\tilde{u}_0), v_\e\right)_{L_2(\tilde{S}\cap B_{bR_2\e}(\tilde{M}_{k,\bot}^\e))}\bigg|
\\
&
\leqslant C\e^{\frac{1}{2}}\|\tilde{u}_0\|_{W_2^1(\Om)}\|v_\e\|_{W_2^1(\Om^\e)}.
\end{aligned}
\end{equation}

Пусть $x\in S\cap B_{bR_2\e}(M_{k,\bot}^\e)$ -- произвольная точка, $x^\bot$ -- её проекция на касательную гиперплоскость к поверхности $S$ в точке $M_{k,\bot}^\e$. Ясно, что $|\xi'|=|x^\bot-M_{k,\bot}^\e|\e^{-1}$. Так как поверхность $S$ гладкая, а линейный размер куска $S\cap B_{bR_2\e}(M_{k,\bot}^\e)$ порядка $O(\e)$, то верно следующее неравенство:
\begin{equation*}
\left|\frac{|\xi'|}{\e R_2}- \frac{|x-M_{k,\bot}^\e|}{\e R_2}\right|=
\left|\frac{|x^\bot-M_{k,\bot}^\e|}{\e R_2}- \frac{|x-M_{k,\bot}^\e|}{\e R_2}\right|\leqslant C\e,
\end{equation*}
где константа $C$ не зависит от $\e$, $k\in \mathbb{M}^\e$ и $x\in S\cap B_{b R_2\e}(M_{k,\bot}^\e)$. Учитывая последнюю оценку и определение функции $\a^\e$, теперь видим, что
\begin{equation}\label{6.7c}
\begin{aligned}
 \bigg|\sum\limits_{k\in\mathbb{M}^\e}
&\frac{\eta^{n-1}|\p\om_{k,\e}|}{R_2^{n-1}}  \left( \z\left(\frac{|\xi'|}{R_2}\right)a(\,\cdot\,,\tilde{u}_0), v_\e\right)_{L_2(\tilde{S}\cap B_{bR_2\e}(\tilde{M}_{k,\bot}^\e))}
\\
&-(\a^\e a(\,\cdot\,,\tilde{u}_0), v_\e)_{L_2(\tilde{S})}\bigg|
\leqslant  C\e \|a(\,\cdot\,,\tilde{u}_0)\|_{L_2(\tilde{S})}\|v_\e\|_{L_2(\tilde{S})}
\\
&\hphantom{-(\a^\e a(\,\cdot\,,\tilde{u}_0), v_\e)_{L_2(\tilde{S})}|..
}
\leqslant  C\e \| \tilde{u}_0 \|_{W_2^1(\Om)}\|v_\e\|_{L_2(\Om^\e)}.
\end{aligned}
\end{equation}
Эта оценка вместе с (\ref{5.5}), (\ref{6.7a})
приводит к неравенству
\begin{equation}\label{6.7b}
 \Big|(a(\,\cdot\,,\tilde{u}_0),v_\e)_{L_2(\p \theta^\e)}
+ (\a^\e a(\,\cdot\,,\tilde{u}_0),v_\e)_{L_2(\tilde{S})}\Big|
\leqslant C\e^{\frac{1}{2}}\|\tilde{u}_0\|_{W_2^1(\Om)}\|v_\e\|_{W_2^1(\Om^\e)}.
\end{equation}
Отсюда уже в силу леммы~\ref{lm3.6} получаем:
\begin{equation*}
\left|(a(\,\cdot\,,\tilde{u}_0),v_\e)_{L_2(\p \theta^\e)} + (\a^0 a(\,\cdot\,,\tilde{u}_0),v_\e)_{L_2(\tilde{S})}\right|\leqslant C\big(\e^{\frac{1}{2}} +\kappa(\e)\big) \|f\|_{L_2(\Omega)}\|v_\e\|_{W_2^1(\Omega_\e)}.
\end{equation*}
Из последнего неравенства и (\ref{5.2}), (\ref{5.3}) следует
\begin{equation*}
\|v_\e\|_{W_2^1(\Om^\e)}\leqslant C\big(\e^{\frac{1}{2}} +\kappa(\e)\big)\|f\|_{L_2(\Om)}.
\end{equation*}

Оценим теперь норму $u_\e-u_0$. Используя последнее неравенство и лемму \ref{lm3.7}, выводим оценку:
\begin{equation*}
\|u_\e-u_0\|_{W_2^1(\Om^\e)}\leqslant\|v_\e\|_{W_2^1(\Om^\e)}+\|\tilde{u}_0-u_0\|_{W_2^1(\Om^\e)}
\leqslant C\big(\e^{\frac{1}{2}}+\kappa(\e)\big)\|f\|_{L_2(\Om)}.
\end{equation*}
Теорема \ref{th2} доказана.

\section{Сходимость в $L_2$-норме}

Настоящий параграф посвящён доказательству теорем~\ref{th3},~\ref{th4}. В доказательстве мы воспользуемся подходом, который применялся в работах \cite{Sen1}, \cite{Sen2}, \cite{Pas1}, \cite{Pas2} для вывода аналогичных утверждений.  А именно, ключевым является следующий факт, справедливый для произвольного рефлексивого банахового пространства: если для некоторого элемента $v$ этого пространства и любого линейного функционала $\cB$ на нем выполнена оценка $
|\cB v|\leqslant  C\|\cB\|$
с константой $C$, не зависящей от $\cB$, то верно $\|v\|\leqslant C$. В нашем случае таким пространством является $L_2(\Om^\e)$, а в качестве функции $v$ берётся функция $u_\e-u_0$, где $u_0$ -- решение соответствующей из усреднённых задач. Мы будем доказывать оценку
\begin{equation}\label{7.2}
\big|(u_\e-u_0,h)_{L_2(\Om^\e)}\big|\leqslant C\big(\vk_1(\e)\|f\|_{L_2(\Om^\e)}+ \vk_2(\e)\|f\|_{L_2(\tht^\e)}\big)\|h\|_{L_2(\Om^\e)}
\end{equation}
для произвольной функции $h\in L_2(\Om)$ и некоторыми функциями $\vk_1(\e)$, $\vk_2(\e)$,  стремящимися к нулю при $\e\to+0$;
здесь и всюду далее через  $C$ обозначаем несущественные константы, не зависящие от $\e$, $f$, $h$, пространственных переменных и функции $V_0$, которая будет введена ниже.
 Отсюда  будет следовать  неравенство для $v_\e$:
\begin{equation}\label{7.3}
\|v_\e\|_{L_2(\Om^\e)}\leqslant C\big(\vk_1(\e)\|f\|_{L_2(\Om^\e)}+ \vk_2(\e)\|f\|_{L_2(\tht^\e)}\big) \|f\|_{L_2(\Om^\e)},
\end{equation}
из которого уже будут вытекать утверждения теорем~\ref{th3},~\ref{th4}.

Пусть $h$ -- произвольная функция из $L_2(\Om^\e)$. Продолжим её нулём внутрь полостей $\tht^\e$ и рассмотрим краевую задачу
\begin{equation}\label{7.5}
\begin{gathered}
\bigg(-\sum\limits_{i,j=1}^n\frac{\p}{\p x_i}A_{ij}\frac{\p}{\p x_j}-\sum\limits_{j=1}^n
\frac{\p\ }{\p x_j}\overline{A_j}
+\overline{A_0}- \overline{\lambda}\bigg) V_0=h\quad\text{в}\quad\Om,
\\
V_0=0\quad\text{на}\quad\p\Om.
\end{gathered}
\end{equation}
Так как $A_j\in W_\infty^1(\Om)$, то согласно лемме~\ref{lm3.2} такая задача однозначно разрешима в $W_2^2(\Om)$. Кроме того, верна оценка
\begin{equation}\label{7.10}
\|V_0\|_{W_2^2(\Om)}\leqslant C\|h\|_{L_2(\Om^\e)}.
\end{equation}

\subsection{Доказательство теоремы~\ref{th3}}
Функция   $v_\e$, являющаяся решением задачи (\ref{7.7}), очевидно принадлежат пространству $W_2^2(\Om^\e)$. С учётом этого факта умножим уравнение в задаче (\ref{7.5}) на $v_\e$ скалярно в $L_2(\Om^\e)$ и дважды проинтегрируем по частям, учитывая краевую задачу (\ref{7.7}). Тогда получим следующее равенство:
\begin{equation}\label{7.9}
\begin{aligned}
(v_\e,h)_{L_2(\Om^\e)}=&
-\bigg(v_\e,\bigg(\frac{\p\ }{\p n}+\sum\limits_{j=1}^{n}\overline{A_j} \nu_j\bigg)V_0\bigg)_{L_2(\p\tht^\e)} + \left(\frac{\p v_\e}{\p n},V_0\right)_{L_2(\p\tht^\e)}
\\
=&-\bigg(v_\e,\bigg(\frac{\p\ }{\p n}+\sum\limits_{j=1}^{n}\overline{A_j} \nu_j\bigg)V_0\bigg)_{L_2(\p\tht^\e)}
-\left(\frac{\p u_0}{\p n},V_0\right)_{L_2(\p\tht^\e)}
\\
&-\left(a(\,\cdot\,,u_\e),V_0\right)_{L_2(\p\tht^\e)}
\\
=&-\bigg(v_\e,\bigg(\frac{\p\ }{\p n}+\sum\limits_{j=1}^{n}\overline{A_j} \nu_j\bigg)V_0\bigg)_{L_2(\p\tht^\e)}
-\left(\frac{\p u_0}{\p n},V_0\right)_{L_2(\p\tht^\e)}
\\
&-\left(a(\,\cdot\,,u_\e)-a(\,\cdot\,,u_0),V_0\right)_{L_2(\p\tht^\e)}
-\left(a(\,\cdot\,,u_0),V_0\right)_{L_2(\p\tht^\e)}.
\end{aligned}
\end{equation}
Оценим правую часть этого равенства.

Проинтегрируем по частям аналогично  (\ref{4.0}):
 \begin{align*}
 \int\limits_{B_1^k\setminus \om_k^\e} v_\e \overline{h}\,dx=&
\int\limits_{B_1^k\setminus \om_k^\e}v_\e \bigg(-\sum\limits_{i,j=1}^n\frac{\p}{\p x_i}A_{ij}\frac{\p}{\p
x_j}-\sum\limits_{j=1}^n
\frac{\p}{\p x_j}A_j
+A_0-\l\bigg)\overline{V_0} \,dx
\\
=&\int\limits_{\p\om_k^\e}v_\e \overline{\bigg(\frac{\p\ }{\p \mathrm{n}}
+\sum\limits_{j=1}^{n} A_j \nu_j
\bigg)V_0}\,ds-\int\limits_{\p B_1^k}v_\e \overline{\bigg(\frac{\p\ }{\p \mathrm{n}}
+\sum\limits_{j=1}^{n} A_j \nu_j
\bigg)V_0}  \,ds
\\
&+\sum\limits_{i,j=1}^n\int\limits_{B_1^k\setminus \om_k^\e} A_{ij}\frac{\p \overline{V_0}}{\p
x_j}\frac{\p v_\e}{\p x_i}\,dx
+\sum\limits_{j=1}^n\int\limits_{B_1^k\setminus \om_k^\e} A_j\frac{\p
v_\e}{\p x_j}\overline{V}_0 \,dx
\\
&
+\int\limits_{B_1^k\setminus \om_k^\e} (A_0-\l) v_\e\overline{V}_0\,dx.
\end{align*}
Далее проинтегрируем по частям с учётом краевых условий в (\ref{4.3}):
\begin{align*}
0=&\sum\limits_{i=1}^n\int\limits_{B_{b_{*}}^k\setminus B_1^k} v_\e \overline{\sum\limits_{j=1}^{n}\bigg(A_{ij}\frac{\p V_0}{\p x_j}+A_i V_0\bigg)} \Delta W_{k,i}^\e \,dx
\\
=& -\int\limits_{\p B_1^k}v_\e \overline{\bigg(\frac{\p\ }{\p \mathrm{n}}
+\sum\limits_{j=1}^{n} A_j \nu_j
\bigg)V_0}  \,ds
\\
&- \sum\limits_{i=1}^n\int\limits_{B_{b_{*}}^k\setminus B_1^k} \nabla W_{k,i}^\e \nabla v_\e \overline{\sum\limits_{j=1}^{n}\bigg(A_{ij}\frac{\p V_0}{\p x_j}+A_i V_0\bigg)} \,dx.
\end{align*}
Полученные соотношения позволяют выразить первое слагаемое в правой части (\ref{7.9}) следующим образом:
\begin{align*}
-&\bigg(v_\e,\bigg(\frac{\p\ }{\p n}+\sum\limits_{j=1}^{n}\overline{A_j} \nu_j\bigg)V_0\bigg)_{L_2(\p\tht^\e)}= - \sum\limits_{k\in\mathds{M}^\e} \int\limits_{B_1^k\setminus \om_k^\e} v_\e \overline{h}\,dx
\\
&
+ \sum\limits_{k\in\mathds{M}^\e} \left(
\sum\limits_{i,j=1}^n\int\limits_{B_1^k\setminus \om_k^\e} A_{ij}\frac{\p \overline{V_0}}{\p
x_j}\frac{\p v_\e}{\p x_i}\,dx
+\sum\limits_{k\in\mathds{M}^\e} \sum\limits_{j=1}^n\int\limits_{B_1^k\setminus \om_k^\e} A_j\frac{\p
v_\e}{\p x_j}\overline{V}_0 \,dx \right)
\\
&
+\sum\limits_{k\in\mathds{M}^\e}  \int\limits_{B_1^k\setminus \om_k^\e} (A_0-\l) v_\e\overline{V}_0\,dx
\\
& +\sum\limits_{k\in\mathds{M}^\e}\sum\limits_{i=1}^n\int\limits_{B_{b_{*}}^k\setminus B_1^k} \nabla W_{k,i}^\e \nabla v_\e \overline{\sum\limits_{j=1}^{n}\bigg(A_{ij}\frac{\p V_0}{\p x_j}+A_i V_0\bigg)} \,dx.
\end{align*}
Эта формула, оценка (\ref{7.10}) и лемма~\ref{lm3.3} приводят к неравенству:
\begin{equation}\label{7.11}
\left|\bigg(v_\e,\bigg(\frac{\p\ }{\p n}+\sum\limits_{j=1}^{n}\overline{A_j} \nu_j\bigg)V_0\bigg)_{L_2(\p\tht^\e)}\right|\leqslant C \big(\e\eta +\e^\frac{1}{2}\eta^{\frac{n}{2}} \big) \|v_\e\|_{W_2^1(\Om^\e)} \|h\|_{L_2(\Om^\e)}.
\end{equation}

Умножим уравнение в (\ref{2.12}) на $V_0$ скалярно в $L_2(\tht^\e)$ и однократно проинтегрируем по частям:
\begin{align*}
\left(\frac{\p u_0}{\p n},V_0\right)_{L_2(\p\tht^\e)} =&
\sum\limits_{i,j=1}^n\int\limits_{\tht^\e} A_{ij}\frac{\p u_0}{\p
x_j}\frac{\p\overline{V}_0}{\p x_i}\,dx
+\sum\limits_{j=1}^n\int\limits_{\tht^\e} A_j\frac{\p
u_0}{\p x_j}\overline{V}_0 \,dx
\\
&
+\int\limits_{\tht^\e} (A_0-\l) u_0\overline{V}_0\,dx - (f,V_0)_{L_2(\tht^\e)}.
\end{align*}
Применение теперь леммы~\ref{lm3.3} и оценок~(\ref{7.10}),~(\ref{4.7})
даёт следующую оценку:
\begin{equation}\label{7.12}
\begin{aligned}
\left| \left(\frac{\p u_0}{\p n},V_0\right)_{L_2(\p\tht^\e)} \right| \leqslant &C  \big(\e^2\eta^2 +\e \eta^n \big) \|f\|_{L_2(\Om^\e)} \|h\|_{L_2(\Om^\e)}
\\
&+C\big(\e\eta +\e^\frac{1}{2}\eta^{\frac{n}{2}} \big)  \|f\|_{L_2(\tht^\e)} \|h\|_{L_2(\Om^\e)}.
\end{aligned}
\end{equation}
В случае $a\equiv 0$ этой оценки, (\ref{7.11}) и (\ref{2.17}) достаточно, чтобы оценить правую часть (\ref{7.9}) и получить неравенство (\ref{7.3}) с
\begin{equation*}
\vk_1(\e)=\e^2\eta^2(\e)+\e \eta^n(\e),\qquad  \vk_2(\e)=\e\eta(\e)+\e^\frac{1}{2}\eta^{\frac{n}{2}}(\e),
\end{equation*}
что уже приводит к (\ref{2.20}).

Пусть  $a\not\equiv 0$ и $\eta(\e)\to+0$ при $\e\to+0$. Из второго условия в (\ref{2.2}) выводим:
\begin{equation*}
\big|a(x,u_\e)-a(x,u_0)\big|\leqslant C|v_\e|.
\end{equation*}
Теперь третье слагаемое в правой части (\ref{7.9}) легко оценивается с помощью леммы~\ref{lm3.3} и (\ref{3.23}), (\ref{12.21}), (\ref{7.10}):
\begin{equation}\label{7.13}
\begin{aligned}
\big|\left(a(\,\cdot\,,u_\e)-a(\,\cdot\,,u_0),V_0\right)_{L_2(\p\tht^\e)} &\big|
\leqslant C\|v_\e\|_{L_2(\p\tht^\e)} \|V_0\|_{L_2(\p\tht^\e)}
\\
\leqslant & C \big(\e\eta+ \eta^{n-1}\big)\|v_\e\|_{W_2^1(\Om^\e)}\|V_0\|_{W_2^1(\Om)}
\\
\leqslant & C\big(\e^2\eta^2
+ \eta^{2(n-1)}\big)\|f\|_{L_2(\Om)} \|h\|_{L_2(\Om^\e)}.
\end{aligned}
\end{equation}

В \cite[Лем. 3.3]{25} была доказана оценка, из которой для произвольной $u\in W_2^2(\Om)$ следует, что
\begin{align*}
\|u\|_{L_2(\tht^\e)}^2 \leqslant C \bigg(&\e\eta  \sum\limits_{k\in\mathds{M}^\e}  \|\nabla u\|_{L_2(B_{b R_2 \e}(M_k^\e)\setminus \om_k^\e)}^2
\\
&+ \e^{-1}\eta^{n-1} \sum\limits_{k\in\mathds{M}^\e}\|u\|_{L_2(L_2(B_{b R_2 \e}(M_k^\e)\setminus \om_k^\e)}^2\bigg).
\end{align*}
Аналогично выводу (\ref{3.7}) из   (\ref{7.18})  получаем:
\begin{equation*}
\|u\|_{L_2(\tht^\e)}^2 \leqslant C \big(\e^2\eta  +\eta^{n-1}\big)
\|u\|_{W_2^2(\Om)}^2.
\end{equation*}
Эта оценка и (\ref{12.21}), (\ref{7.10}), (\ref{4.7}) позволяют теперь оценить последнее слагаемое в правой части (\ref{7.9}):
\begin{equation*}
\big|\left(a(\,\cdot\,,u_0),V_0\right)_{L_2(\p\tht^\e)}\big| \leqslant
C\big(\e^2\eta  +\eta^{n-1}\big) \|f\|_{L_2(\Om)} \|h\|_{L_2(\Om^\e)}.
\end{equation*}
Из последней оценки, (\ref{7.13}), (\ref{7.12}), (\ref{7.11}), (\ref{7.12}) вытекает неравенство (\ref{7.3}) с
\begin{equation*}
\vk_1(\e)=\e^2\eta(\e)+ \eta^{n-1}(\e),\qquad  \vk_2(\e)=\e\eta(\e)+\e^\frac{1}{2}\eta^{\frac{n}{2}}(\e),
\end{equation*}
что означает справедливость (\ref{2.4}). Теорема доказана.

\subsection{Доказательство теоремы~\ref{th4}}

Умножим  уравнение в задаче (\ref{7.5}) на функцию $u_\e$ скалярно в $L_2(\Om^\e)$ и дважды проинтегрируем по частям, учитывая уравнение в (\ref{2.5}). Тогда аналогично
 (\ref{7.9}) получаем:
\begin{equation}\label{7.26}
\begin{aligned}
(u_\e,h)_{L_2(\Om^\e)}=&-\bigg(u_\e,\bigg(\frac{\p\ }{\p n}+\sum\limits_{j=1}^{n}\overline{A_j} \nu_j\bigg)V_0\bigg)_{L_2(\p\tht^\e)}
\\
&-\left(a(\,\cdot\,,u_\e) -a(\,\cdot\,,u_0),V_0\right)_{L_2(\p\tht^\e)}
\\
&
-\left(a(\,\cdot\,,u_0),V_0\right)_{L_2(\p\tht^\e)} +(f,V_0)_{L_2(\Om^\e)}.
\end{aligned}
\end{equation}
Далее умножим уравнение в задаче (\ref{7.5}) на $u_0$ скалярно в $L_2(\Om)$ и вновь дважды проинтегрируем по частям, учитывая краевую задачу (\ref{2.13}), (\ref{2.15}):
\begin{equation*}
(u_0,h)_{L_2(\Om^\e)}=(u_0,h)_{L_2(\Om)}=
 (\a^0 a(\,\cdot\,, u_0),V_0)_{L_2(S)} +(f,V_0)_{L_2(\Om)}.
\end{equation*}
Вычтем это равенство из (\ref{7.26}) и после элементарных преобразований получаем:
\begin{equation}\label{7.19}
\begin{aligned}
(v_\e,h)_{L_2(\Om^\e)}=&-\bigg(u_\e,\bigg(\frac{\p\ }{\p n}+\sum\limits_{j=1}^{n}\overline{A_j} \nu_j\bigg)V_0\bigg)_{L_2(\p\tht^\e)}
-\left(\frac{\p u_0}{\p n},V_0\right)_{L_2(\p\tht^\e)}
\\
&-\left(a(\,\cdot\,,u_\e) -a(\,\cdot\,,u_0),V_0\right)_{L_2(\p\tht^\e)}
-(f,V_0)_{L_2(\tht^\e)}
\\
&-(\a^0 a(\,\cdot\,,u_0),V_0)_{L_2(\tilde{S})}
-\left(a(\,\cdot\,,u_0),V_0\right)_{L_2(\p\tht^\e)}
\\
&-(\a^0 a(\,\cdot\,,u_0),V_0)_{L_2(S)} + (\a^0 a(\,\cdot\,,u_0),V_0)_{L_2(\tilde{S})}.
\end{aligned}
\end{equation}
Как в доказательстве теоремы~\ref{th3}, оценим правую часть этого равенства. Для первых трёх слагаемых в правой части верны неравенства (\ref{7.11}), (\ref{7.12}), (\ref{7.13}). Поэтому оценки требует только оставшиеся пять слагаемых, что и будем нашей основной целью в дальнейших вычислениях.

Из леммы~\ref{lm3.3} и (\ref{7.10}) сразу выводим:
\begin{equation}\label{7.27}
\big|(f,V_0)_{L_2(\tht^\e)}\big| \leqslant C\big(\e\eta +\e^\frac{1}{2}\eta^{\frac{n}{2}} \big)  \|f\|_{L_2(\tht^\e)} \|h\|_{L_2(\Om^\e)}.
\end{equation}
Сумму
\begin{equation*}
-(\a^0 a(\,\cdot\,,u_0),V_0)_{L_2(\tilde{S})}
-\left(a(\,\cdot\,,u_0),V_0\right)_{L_2(\p\tht^\e)}
\end{equation*}
 в правой части (\ref{7.19}) будем оценивать также, как это было сделано для аналогичного выражения в доказательстве теоремы~\ref{th2}: необходимо лишь заменить $v_\e$ на $V_0$, а $\tilde{u}_0$  на $u_0$.  При этом следует дополнительно использовать оценки (\ref{7.23}) и (\ref{7.21}). В результате в правых частях оценок, аналогичных (\ref{5.4}), (\ref{5.5}), (\ref{6.7a}), возникают выражения $C\e \|f\|_{L_2(\Om)}\|h\|_{L_2(\Om^\e)}$. Оценка (\ref{6.7c}) остаётся без изменений. В итоге приходим к следующему аналогу оценки (\ref{6.7b}):
\begin{equation*}%\l%abel{7.24}
\Big|(a(\,\cdot\,,\tilde{u}_0),V_0)_{L_2(\p \theta^\e)}
+(\a^\e a(\,\cdot\,,\tilde{u}_0),V_0)_{L_2(\tilde{S})}\Big|
\leqslant C\e \|f\|_{L_2(\Om)}\|h\|_{L_2(\Om^\e)}.
\end{equation*}
Пользуясь теперь леммой~\ref{lm3.6}, получаем неравенство:
\begin{equation}\label{7.28}
\Big|\left(a(\,\cdot\,,\tilde{u}_0),V_0\right)_{L_2(\p\tht^\e)} + (\a^0 a(\,\cdot\,,\tilde{u}_0),V_0)_{L_2(\tilde{S})}\Big| \leqslant C(\e+\kappa(\e)) \|f\|_{L_2(\Om)}\|h\|_{L_2(\Om^\e)}.
\end{equation}

Разность последних двух слагаемых в правой части (\ref{7.19}) представим в виде следующего интеграла по аналогии с (\ref{12.18}):
\begin{equation}\label{7.29}
\begin{aligned}
(\a^0a(\,\cdot\,,u_0),V_0)_{L_2(\tilde{S})}-&(\a^0 a(\,\cdot\,,u_0),V_0)_{L_2(S)}
\\
=& \int\limits_{S} a^0(x) a(x,u_0(x))\overline{V_0(x)}\bigg|_{\tau=\e(2bR_2+R_0)} \cos(\nu,\tilde{\nu})\,ds
\\
&- \int\limits_{S} a^0(x) a(x,u_0(x))\overline{V_0(x)} \,ds
\\
=&\int\limits_{S} a^0(x) a(x,u_0(x))\overline{V_0(x)}\bigg|_{\tau=\e(2bR_2+R_0)} \big(\cos(\nu,\tilde{\nu})-1\big)\,ds
\\
&+  \int\limits_{S} ds\int\limits_{0}^{\e(2bR_2+R_0)}a^0(x) \frac{\p\ }{\p\tau} a(x,u_0(x))\overline{V_0(x)} \,ds.
 \end{aligned}
\end{equation}
Верна оценка, аналогичная (\ref{12.17}):
\begin{equation*}
|\cos(\nu,\tilde{\nu})-1|\leqslant C\e.
\end{equation*}
Используя эту оценку, (\ref{4.7a}), (\ref{7.10}), (\ref{12.21}), (\ref{3.7}), из (\ref{7.29}) выводим:
\begin{equation*}
\big| (\a^0a(\,\cdot\,,u_0),V_0)_{L_2(\tilde{S})}-(\a^0 a(\,\cdot\,,u_0),V_0)_{L_2(S)}
\big| \leqslant C\e \|f\|_{L_2(\Om)}\|h\|_{L_2(\Om^\e)}.
\end{equation*}
Из этого неравенства, (\ref{7.28}), (\ref{7.27}) и упомянутых выше улучшенных аналогов  (\ref{5.4}), (\ref{5.5}), (\ref{6.7a}) уже следует оценка
 (\ref{7.2}) с
\begin{equation*}%\l%abel{7.25}
\vk_1(\e)=\e+ \kappa(\e),\qquad  \vk_2(\e)=\e\eta(\e)+\e^\frac{1}{2}\eta^{\frac{n}{2}}(\e),
\end{equation*}
из которой вытекает неравенство (\ref{7.3}).  Теорема доказана.

\section*{Благодарности}

Исследование  выполнено за счет гранта Российского научного фонда
(проект № 20-11-19995).

\end{document}